\pdfoutput=1
\documentclass[letterpaper, 10 pt, conference]{ieeeconf}  

\IEEEoverridecommandlockouts                              

\usepackage{mathrsfs}
\usepackage[font={small}]{caption}
\usepackage{scalerel}
\usepackage{url}
\usepackage{bbold}
\usepackage{amsfonts}
\usepackage{amsmath,amssymb,amsfonts}
\usepackage{graphicx}
\usepackage{wrapfig}
\usepackage{mathrsfs} 
\usepackage{algorithm,algorithmic}
\usepackage{array}
\usepackage{times}
\usepackage{url}
\usepackage{cite}
\usepackage[thinc]{esdiff}
\newcommand{\citet}[1]{\cite{#1}}
\usepackage{upgreek}
\usepackage{float}
\usepackage{longtable}
\usepackage{color}
\usepackage{wasysym}
\usepackage{grffile}
\usepackage{stackengine}
\usepackage[capitalize,noabbrev]{cleveref}

\allowdisplaybreaks

\newcommand{\bb}{\mathbf{b}}

\newcommand{\eb}{\mathbf{e}}

\newcommand{\vb}{\mathbf{v}}
\newcommand{\ub}{\mathbf{u}}
\newcommand{\wb}{\mathbf{w}}
\newcommand{\xb}{\mathbf{x}}
\newcommand{\yb}{\mathbf{y}}

\newcommand{\Ab}{\mathbf{A}}
\newcommand{\Bb}{\mathbf{B}}
\newcommand{\Cb}{\mathbf{C}}

\newcommand{\Gb}{\mathbf{G}}
\newcommand{\Hb}{\mathbf{H}}

\newcommand{\Kb}{\mathbf{K}}
\newcommand{\Lb}{\mathbf{L}}
\newcommand{\Mb}{\mathbf{M}}
\newcommand{\Nb}{\mathbf{N}}
\newcommand{\Pb}{\mathbf{P}}

\newcommand{\Rb}{\mathbf{R}}

\newcommand{\Xb}{\mathbf{X}}

\newcommand{\Xc}{\mathcal{X}}

\newcommand{\norm}[1]{\left\lVert#1\right\rVert}

\newtheorem{theorem}{Theorem}

\newtheorem{assumption}{Assumption}

\newtheorem{definition}{Definition}

\newtheorem{lemma}[theorem]{Lemma}

\newtheorem{remark}{Remark}

\title{\bf Regret Analysis of Distributed Online Control for LTI Systems with Adversarial Disturbances}

\author{Ting-Jui Chang and Shahin Shahrampour  
\thanks{T.J. Chang and S. Shahrampour are with the Department of Mechanical and Industrial Engineering, Northeastern University, Boston, MA 02115, USA. 
{\tt\footnotesize email:\{chang.tin,s.shahrampour\}@northeastern.edu}.}%
}

\begin{document}

\maketitle
\thispagestyle{plain}
\pagestyle{plain}

\begin{abstract}
  This paper addresses the distributed online control problem over a network of linear time-invariant (LTI) systems (with possibly unknown dynamics) in the presence of adversarial perturbations. There exists a global network cost that is characterized by a {\it time-varying} convex function, which evolves in an adversarial manner and is sequentially and partially observed by local agents. The goal of each agent is to generate a control sequence that can compete with the best {\it centralized} control policy in hindsight, which has access to the global cost. This problem is formulated as a {\it regret} minimization. For the case of known dynamics, we propose a fully distributed disturbance feedback controller that guarantees a regret bound of $O(\sqrt{T}\log T)$, where $T$ is the time horizon. For the unknown dynamics case, we design a distributed explore-then-commit approach, where in the exploration phase all agents jointly learn the system dynamics, and in the learning phase our proposed control algorithm is applied using each agent system estimate. We establish a regret bound of $O(T^{2/3} \text{poly}(\log T))$ for this setting.
\end{abstract}

\section{Introduction}
In recent years, there has been a significant interest  in studying problems at the interface of control and machine learning, and modern statistical and online optimization tools have paved the way to rethink classical control problems. Consider the linear quadratic regulator (LQR) problem, where for a given LTI system, an optimal control problem is defined with a time-invariant quadratic function of control-state pairs. It is well-known that when the underlying dynamics is known, the optimal controllers can be computed by solving the Riccati equations for both cases of finite-horizon and infinite-horizon problems.

While many classical optimal/adaptive control problems serve as good examples of real-world phenomena, in certain frameworks, they may fail to provide adaptive solutions for non-stationary environments that have these three characteristics. (I) The environment can vary in an unpredictable manner, which makes the cost metrics {\it time-varying} and {\it unknown} in advance. (II) The dynamical system may be driven with a noise sequence that does not follow any statistical assumption (i.e., it may be chosen in an {\it adversarial} manner). (III) The parameters of the underlying dynamics may be {\it unknown}. Due to these challenges, one cannot design an optimal controller sequence that depends on future information (e.g., future cost or noise instances).

In these cases, there is a need to develop {\it online} controllers, the performance of which is typically measured by the notion of {\it regret}, defined as the difference between the accumulated cost of the online controller and that of the best control policy in hindsight. In online control, the  goal is to keep the regret bound sub-linear with respect to the time horizon, which implies that the time-averaged performance of the online controller mimics that of the best policy asymptotically. Though the regret measure originates from the  online optimization framework, it has been shown that there exist direct connections between regret and the stability properties of linear and non-linear systems \cite{karapetyan2022implications,nonhoff2023relation} studied in control theory, which further justifies the study of regret for online control. For challenges posed in (I-II), recent works often transform the respective online control problem to an online learning and leverage  online optimization techniques to design online controllers (e.g., semidefinite programming (SDP) relaxation for time-varying LQR \cite{cohen2018online} and noise feedback policy design for time-varying convex costs in \cite{agarwal2019online}). (III) is typically addressed via adaptive control, where the learner must maintain a good balance between exploration (estimating the system parameters while keeping the cost small enough) and exploitation (constructing a controller based on the system estimate) \cite{abbasi2011regret,ibrahimi2012efficient,dean2018regret,cohen2019learning}. It is also worth noting that for classical control algorithms making control decisions in an online manner (e.g., model predictive control), several types of update rules can be considered as special cases of online learning framework \cite{wagener2019online}.

On the other hand, the  computational cost of complex control problems motivates the design of distributed control algorithms with modest local computational costs. A well-known setup is the case where the system is composed of a network of sub-systems with a joint task and decoupled dynamics, e.g., microsatellite clusters \cite{burns2000techsat,schaub2000spacecraft}, unmanned aerial vehicles \cite{buzogany1993automated,wolfe1996decentralized} and mobile robotics \cite{yamaguchi1998cooperative,yamaguchi2001distributed,yang2022collaborative}. Motivated by these practical considerations, in this paper,  we study the ``distributed" online control problem over a network of LTI systems (with possibly unknown dynamics) in the presence of adversarial perturbations. Each agent in the network is modeled as an LTI system. There exists a global network cost that is characterized by a {\it time-varying} convex function, which evolves in an adversarial manner and is sequentially and partially observed by local agents. The goal of each agent is to design an online control sequence (using partial information) that attains a comparable performance to that of the best centralized policy in hindsight, which keeps the regret small. Our main contributions are summarized as follows
\begin{enumerate}
    \item For the case of known system dynamics, we build on the disturbance feedback controller (DFC) proposed in \cite{agarwal2019online} to develop a distributed algorithm called distributed DFC (D-DFC). At each round, the network performs a distributed online gradient descent (D-OGD) step to update each agent control parameters, which are used together with the past noises to determine the action of the next round. For the unknown dynamics case, we propose a two-phase distributed algorithm. In the exploration phase, agents collect the data and jointly estimate the system parameters in a distributed fashion. During this phase, no learning is achieved and the actions are determined stochastically. In the exploitation phase, D-OGD is applied to update the control parameters, but the action of each agent is determined based on its own noise and system estimates.

    \item For the case of known system dynamics, we show that via the reparameterization of D-DFC, the distributed online control problem can be transformed into a distributed online learning (D-OL) problem with memory. Then, by appropriately choosing the hyper-parameters (learning rate and the window length of past noises considered by the current action), we establish a regret bound of $O(\sqrt{T}\log T)$ (Theorem \ref{T: Regret bound (known case)}), which recovers the regret rate of its centralized counterpart in \cite{agarwal2019online}.

    \item For the unknown dynamics case, to ensure a sub-linear regret bound, it is important to maintain a good balance between the exploration and exploitation phases as they play conflicting roles against each other. 
    A key difference compared to the known dynamics case is that the local costs used by D-OGD are defined based on each agent system and noise estimates. 
    This deviates the framework from the usual distributed setting where the sum of local costs is equal to the global cost. Therefore, the analysis of D-OL with memory is not directly applicable. To tackle this challenge, we provide theoretical results that quantify the effect of different hyper-parameters on the function value (Lemma 17). Combining the analysis mentioned above, we prove a regret bound of $O\big(T^{2/3}\text{poly}(\log T)\big)$ (Theorem \ref{T: Regret bound (unknown case)}). 
    
\end{enumerate}

\section{Related Literature}
\textbf{Online Optimization:} Online control is related to online convex optimization (OCO), where a learner has to select a sequence of actions in real-time for a time-varying convex function  not known in advance. In the last two decades, many OCO algorithms have been developed, including OGD by the seminal work of \cite{zinkevich2003online}, which achieves $O(\sqrt{T})$ regret (see e.g., \cite{cesa2006prediction,hazan2016introduction} for more details). Our work is closely related to the setup of OCO with memory \cite{anava2015online} though we develop a distributed variant of this technique. In summary, OCO investigates pure optimization problems as opposed to dynamical systems. 

\textbf{Regret vs. Stability:} As mentioned previously, the resulting regret of an online control policy can be regarded as an indicator of the stability of the corresponding closed loop system. In the work of \cite{karapetyan2022implications}, it was shown that for linear policies and linear systems subject to adversarial noises, linear regret implies asymptotic stability for both time-invariant and time-varying systems. Conversely, the asymptotic stability (bounded input-bounded state stability and absolute summability of the state transition matrices, respectively) implies linear regret for time-invariant systems (time-varying systems, respectively). Such a relationship was also discussed for non-linear systems in \cite{nonhoff2023relation}.

\textbf{LQR with Online Action Sequence:} A well-structured  optimal control problem is the LQR, where the cost functions are quadratic in state-action pairs. 
When the system dynamics is {\it unknown} and adaptive controllers should be built on the system estimates, there are many works in the literature that prove sub-linear regret bounds \cite{dean2018regret,cohen2019learning,cassel2020logarithmic,simchowitz2020naive,lale2022reinforcement}. Another scenario is that the system might be perturbed with noises that are potentially adversarial. \cite{yu2020power} proposed a model predictive control type method with correct noise predictions for LQR with a time-invariant cost. \cite{zhang2021regret} extended this idea for the case where the quadratic cost is time-varying and the noise prediction is only approximately correct. Both works provided bounds for the dynamic regret.

\textbf{Online Control with Time-Varying Costs:}
We roughly divide the online control literature into two categories, based on whether the system dynamics is known or not.

\textbf{(1) Known Systems:} \cite{cohen2018online} studied the case where the system is perturbed by stochastic noises, and the time-varying costs are quadratic. By formulating the problem as a time-varying SDP, they proved a regret bound of $O(\sqrt{T})$. The setup was later extended to the distributed case and the regret bound of the same order was given in \cite{chang2021distributed}. \cite{agarwal2019online} considered the setup where the noises are chosen adversarily and the costs are generally convex. The authors proposed the disturbance-action policy and parameterized the online control problem as an online learning problem with memory, which achieves a regret bound of $O(\sqrt{T})$. The regret bound was later improved to $O(\text{poly}(\text{log}T))$ for strongly convex costs \cite{agarwal2019logarithmic}. \cite{simchowitz2020improper} further extended the $O(\text{poly}(\text{log}T))$ regret bound to partially observable systems with semi-adversarial disturbances. 

\textbf{(2) Unknown Systems:} \cite{hazan2020nonstochastic} considered a fully observable system with adversarial noises. They showed that by applying the disturbance-action policy computed on noises estimated using the learned model, the regret bound is of $O(T^{2/3})$ for convex functions. The more recent work of \cite{cassel2022rate} considered a similar setup where the noises are i.i.d. Gaussian. They proposed a method where the system estimate is updated throughout the learning process and proved a regret bound of $O(\sqrt{T})$. For partially observable systems, \cite{simchowitz2020improper} provided regret bounds of $O(T^{2/3})$ and $O(\sqrt{T})$ for convex functions (in adversarial noise setting) and strongly-convex functions (in semi-adversarial noise setting), respectively. \cite{lale2020logarithmic} showed the regret bound can be improved to $O(\text{poly}(\text{log}T))$ for strongly convex functions when noises are stochastic. For the distributed setup, the work of \cite{chang2023regret} proposed an $O(T^{2/3}\log T)$ regret bound for the case where the cost functions are quadratic and the noises are Gaussian.

\textbf{System Identification:}
To address the optimal control problem for an unknown system, it is important to have an estimation of the system parameters first. The classical theory of system identification for LTI systems \cite{aastrom1971system,ljung1999system,chen2012identification,goodwin1977dynamic} characterizes the asymptotic properties of the estimators. On the other hand, recent studies focused on developing finite-time guarantees. For systems perturbed by stochastic noises, \cite{dean2019sample} proposed a least-squares estimator to learn the model from multiple trajectories for a fully observable LTI system. These results were later extended to the estimation using a single trajectory \cite{simchowitz2018learning, sarkar2019near}. For partially observable systems, estimators with sample complexities scaling polynomially on the system dimension were provided in the literature \cite{oymak2019non, sarkar2019finite, tsiamis2019finite, simchowitz2019learning}. \cite{fattahi2020learning} further improved the dependency to be poly-logarithmic. For the setup of adversarial noises, least-squares estimators can result in inconsistent estimates. \cite{hazan2020nonstochastic} and \cite{simchowitz2019learning} proposed methods for tackling fully and partially observable systems, respectively. Unlike aforementioned methods that are offline in the sense that the data is collected first, \cite{kowshik2021streaming} proposed an online estimation algorithm which applies stochastic gradient descent in the reverse order of the data sequence using buffers. This idea was later extended to the distributed setup \cite{chang2022distributed} and was also considered for the system identification of generalized linear models \cite{kowshik2021near}.

\textbf{Online Control with Time-Varying Dynamics:} As opposed to the works mentioned previously, which were mainly focused on time-invariant dynamical systems, the authors of \cite{gradu2023adaptive} considered the online control problems with linear time-varying dynamics. They introduced the idea of adaptive regret to capture the adaptation of the controller to changing dynamics. This setup was later extended to safety control where the state and control sequences have to meet certain constraint requirements, and bounds in terms of dynamic regret were presented for fully observable linear systems \cite{zhou2023safe} and partially observable systems \cite{zhou2022safe}.

\section{Preliminaries and Problem Setup}
\subsection{Notation}
\begin{center}
\begin{tabular}{|c||l|}
    \hline
    $[m]$ & The set $\{1,2,\ldots,m\}$ for any integer $m$ \\
    \hline
    $\norm{\cdot}$ & Euclidean (spectral) norm of a vector (matrix)\\
    \hline
    $\norm{\cdot}_F$ & Frobenius norm of a matrix\\
    \hline
   $\Pi_{\Xc}[\cdot]$ & The operator for the projection to set $\Xc$ \\
    \hline
    $[\Ab]_{ij}$ & The entry in the $i$-th row and $j$-th column of $\Ab$\\
    \hline
    $[\Ab]_{i,:}$ & The $i$-th row of $\Ab$\\
    \hline
    $[\Ab]_{:,j}$ & The $j$-th column of $\Ab$\\
    \hline
    $\mathbf{1}$ & The vector of all ones\\
    \hline
    $\eb_i$ & The $i$-th basis vector\\
    \hline
    $\mathbb{1}$ & Indicator function\\
    \hline
\end{tabular}
\end{center}

\subsection{Dynamical System and Cost Model}
We consider a multi-agent network of $m$ identical LTI systems, where the dynamics of agent $i$ is governed by,
$$\xb_{i,t+1} = \Ab\xb_{i,t}+\Bb\ub_{i,t}+\wb_{t},\quad i\in [m],$$
where $\xb_{i,t}\in\mathrm{R}^{d_1}$ and $\ub_{i,t}\in\mathrm{R}^{d_2}$ represent agent $i$ state  and control (or action) at time $t$, respectively. The noise sequence $\{\wb\}$ is assumed to be bounded and {\it unknown} to agents in advance, and it can be chosen in an {\it adversarial} fashion. In this work, we consider two cases: (1) The system parameters $(\Ab,\Bb)$ are known; (2) The parameters $(\Ab,\Bb)$ are unknown, and all agents have to jointly estimate the parameters and decide actions based on their learned system estimates.

The distributed online control problem is as follows: At round $t$, agent $i$ observes the state $\xb_{i,t}$ and applies the action $\ub_{i,t}$. Then, the {\it local function} $c_{i,t}(\cdot,\cdot)$ is revealed to the agent, and the cost $c_{i,t}(\xb_{i,t}, \ub_{i,t})$ is incurred. Instead of being driven by the individual local costs, a distributed online control policy aims at minimizing the network cost $c_t(\cdot,\cdot) = \sum_{i=1}^m c_{i,t}(\cdot,\cdot)$. That is, a successful policy is designed such that the state-control sequence $\{(\xb_i,\ub_i)\}$ of any agent $i\in [m]$ enjoys a good performance on the {\it network} cost, which is consistent with the {\it cooperative} control setting. 

\textbf{Centralized Benchmark and Individual Regret:} To measure the performance of a distributed online control algorithm, consider a centralized benchmark policy $\pi$ for which the finite-time cost of state-action pairs $(\xb^{\pi}_t, \ub^{\pi}_t)$ after $T$ rounds is computed as
\begin{equation}\label{Eq: Centralized benchmark}
    J_T(\pi)=\sum_{t=1}^T c_t(\xb^{\pi}_t, \ub^{\pi}_t).
\end{equation}
The goal of a distributed online control algorithm $\mathcal{A}$ is to compete with the best centralized policy in hindsight, which minimizes $\eqref{Eq: Centralized benchmark}$ over a class of policies $\Pi$. The suboptimality is captured by the individual {\it regret}, which is defined as follows
\begin{align}\label{Eq: Individual Regret}
    \text{Regret}_T^j(\mathcal{A}):=J_T^j(\mathcal{A})-\min_{\pi\in\Pi}J_T(\pi),
\end{align}
for agent $j\in [m]$, where 
\begin{align}
    J_T^j(\mathcal{A}) 
    &=\sum_{t=1}^T c_t(\xb^{\mathcal{A}}_{j,t}, \ub^{\mathcal{A}}_{j,t})=\sum_{t=1}^T\sum_{i=1}^m c_{i,t}(\xb^{\mathcal{A}}_{j,t}, \ub^{\mathcal{A}}_{j,t}).
\end{align}
In above, $(\xb^{\mathcal{A}}_{j,t}, \ub^{\mathcal{A}}_{j,t})$ are the state-action pairs generated by agent $j$ using algorithm $\mathcal{A}$. In this work, the benchmark policy class is defined as the class of strongly stable linear policies (see Definition \ref{D: Strong Stability}). From the definition of regret, a successful distributed online control policy is one that attains a regret upper bound sublinear in $T$, which implies that the temporal average cost of the distributed algorithm approaches to that of the best centralized policy in hindsight asymptotically.

\textbf{Network Structure:} Since the goal is to minimize the network cost, it is necessary for all agents to communicate their local function information with their neighbors in order to {\it approximately} follow the network cost. The network topology governing the communication is defined with an undirected graph $\mathcal{G}=(\mathcal{V},\mathcal{E})$, where $\mathcal{V}=[m]$ denotes the set of nodes (i.e., agents) and $\mathcal{E}$ represents the set of edges. The weighted adjacency matrix is denoted as $\Pb$, where $[\Pb]_{ji}>0$ if agent $j$ communicates with agent $i$; otherwise $[\Pb]_{ji}=0$. $\Pb$ is assumed to be symmetric and doubly stochastic, i.e., all elements of $\Pb$ are non-negative and $\sum_{i=1}^m [\Pb]_{ji}=\sum_{j=1}^m [\Pb]_{ji}=1$. We further assume $\Pb$ is connected and has a positive diagonal. Then, there exists a geometric mixing bound for $\Pb$, such that $\sum_{j=1}^m \left|[\Pb^k]_{ji}-1/m\right|\leq \sqrt{m}\beta^k,\:i\in [m],$ where $\beta$ is the second largest singular value of $\Pb$.

\subsection{Strong Stability and Strong Controllability}
The benchmark policy class is considered as the set of strongly stable linear (i.e., $\ub=-\Kb\xb$) controllers. Following \cite{cohen2018online}, we define the notion of strong stability as follows. 
\begin{definition}\label{D: Strong Stability}(Strong Stability) A linear policy $\Kb$ is $(\kappa, \gamma)$-strongly stable (for $\kappa > 0$ and $0<\gamma\leq 1$) for the LTI system $(\Ab,\Bb)$, if $\norm{\Kb}\leq \kappa$, and there exist matrices $\Lb$ and $\Hb$ such that $\Ab-\Bb\Kb=\Hb\Lb\Hb^{-1}$, with $\norm{\Lb}\leq 1-\gamma$ and $\norm{\Hb},\|\Hb^{-1}\|\leq \kappa$.
\end{definition}
The notion of strong stability provides a quantification of a stable linear policy, in the sense that any stable linear policy is strongly stable with some $\kappa$ and $\gamma$. As mentioned in \cite{cohen2018online}, a sufficient condition ensuring the existence of a strongly stable linear policy is the controllability of the system $(\Ab,\Bb)$, which is characterized by the ability to drive any non-zero initial state to zero through the unperturbed dynamics $\xb_{t+1} = \Ab\xb_{t} + \Bb\ub_{t}$. Similar to the strong stability, the notion of strong controllability is also presented in \cite{cohen2018online} as follows:

\begin{definition}({\it Strong Controllability})
    For a linear system $(\Ab,\Bb)$ and $q\geq 1$, define a matrix $\Cb_{q}\in\mathrm{R}^{d_1\times qd_2}$ as 
    $$\Cb_{q} = [\Bb, \Ab\Bb, \Ab^2\Bb,\ldots, \Ab^{q-1}\Bb].$$
    The system $(\Ab,\Bb)$ is controllable with index $q$ if $\Cb_q$ is full row-rank, and it is $(q,\kappa)$-strongly controllable if $\norm{(\Cb_q\Cb_q^{\top})^{-1}}\leq \kappa$. 
\end{definition}

Compared to the strong stability, which focuses on quantifying the stabilizing property of a given policy $\Kb$, the idea of strong controllability helps with quantifying the effect of a given initial state. For example, in \cite{cohen2018online}, the bound on the accumulated cost of an online LQR problem is given in terms of the initial state. Strong controllability also plays an important role in system identification. For example, \cite{hazan2020nonstochastic} showed that the estimation error of $\Ab$ is related to the smallest singular value of $\Cb_q$ as the estimate of $\Ab$ is derived by solving a perturbed linear equation in terms of $\Cb_q$.

\subsection{Disturbance Feedback Controller}
One challenge preventing the efficient online learning for linear policies of type $\ub=-\Kb\xb$ is the non-convex parameterization, as $\xb_t$ becomes a non-convex function of the policy parameter $\Kb$. To tackle this issue, instead of using the state feedback control using $\Kb$, \cite{agarwal2019online} proposed a disturbance-feedback controller (DFC), where the control is represented as a linear combination of the past noises given the knowledge of a $\Kb$ matrix that incorporates the effect of the current state. In this paper, we extend DFC to the distributed setup, where the linear coefficients of all agents are jointly updated by applying D-OGD. However, let us start by stating the definition of DFC in its centralized form next.

\begin{definition}
    A disturbance feedback controller $(\Mb, \Kb)$ is defined by a matrix $\Kb$ and parameters $\Mb = \{\Mb^{[0]},\ldots,\Mb^{[H-1]}\}$ for an integer $H\geq 1$. At round $t$, the action $\ub_t$ on $\xb_t$ is chosen as follows
    \begin{equation*}
        \ub^{\Kb}_t = -\Kb\xb^{\Kb}_t + \sum_{k=1}^H \Mb^{[k-1]}\wb_{t-k}, 
    \end{equation*}
    where $\wb_t=0$ for $t\leq 0$ by default.
\end{definition}

The parameters $\Mb$ in DFC must be learned as a part of the process. Consider a mechanism that learns these parameters, leading to the time-varying DFC $\ub_t = -\Kb\xb_t + \sum_{k=1}^H \Mb_t^{[k-1]}\wb_{t-k}$. Expanding the dynamics recursively, $\xb_{t+1}$ can be expressed as the following form
\begin{equation*}
\resizebox{\linewidth}{!}
    {
    \begin{minipage}{\linewidth}
    \begin{align*}
        \xb^{\Kb}_{t+1} = (\Ab - \Bb\Kb)^{h+1}\xb^{\Kb}_{t-h} + \sum_{i=0}^{H+h} \Psi^{\Kb,h}_{t,i}(\Mb_{t-h:t}|\Ab,\Bb) \wb_{t-i},
    \end{align*}
    \end{minipage}
    }
\end{equation*}
where the disturbance-state transfer matrix $\Psi^{\Kb,h}_{t,i}(\Mb_{t-h:t}|\Ab,\Bb)$ is defined as follows.
\begin{definition}(Disturbance-state transfer matrix)
For any $t$, $h\leq t$, $i\leq H+h$, the transfer matrix $\Psi^{\Kb,h}_{t,i}(\Mb_{t-h:t}|\Ab,\Bb)$ has the following form
\begin{equation*}
\resizebox{\linewidth}{!}
    {
    \begin{minipage}{\linewidth}
    \begin{align*}
        &\Psi^{\Kb,h}_{t,i}(\Mb_{t-h:t}|\Ab,\Bb)\\ = &(\Ab - \Bb\Kb)^i \mathbb{1}_{i\leq h} + \sum_{j=0}^h (\Ab - \Bb\Kb)^j\Bb\Mb^{[i-j-1]}_{t-j}\mathbb{1}_{i-j\in [1,H]}.
    \end{align*}
    \end{minipage}
    }
\end{equation*}
\end{definition}
In the work of \cite{agarwal2019online}, it is shown that DFC can approximate a linear policy well as the controller takes the past noise sequence into consideration long enough (i.e.,  large enough $H$). For a time-varying DFC $(\Mb_t,\Kb)$, from the above definition we can see that $\xb_{t+1}$ is a function of $\Mb_{1:t}$, which prevents this framework from fitting into the online learning setting, as the number of parameters grows over time. To deal with this issue, we leverage the idea of the surrogate state $\yb^{\Kb}_{t+1}$ \cite{agarwal2019online}, which disregards the dependency on $\xb^{\Kb}_{t-H}$.
\begin{definition}
The surrogate state $\yb^{\Kb}_{t+1}$, the surrogate action $\vb^{\Kb}_{t+1}$ and the resulting cost are defined as follows
\begin{equation*}
\resizebox{\linewidth}{!}
    {
    \begin{minipage}{\linewidth}
    \begin{align*}
        &\yb^{\Kb}_{t+1}(\Mb_{t-H:t}|\Ab,\Bb,\{\wb\})\\ = &\sum_{i=0}^{2H} \Psi^{\Kb,H}_{t,i}(\Mb_{t-H:t}|\Ab,\Bb) \wb_{t-i},\\
        &\vb^{\Kb}_{t+1}(\Mb_{t-H:t}|\Ab,\Bb,\{\wb\})\\ = &-\Kb\yb^{\Kb}_{t+1}(\Mb_{t-H:t}|\Ab,\Bb,\{\wb\}) + \sum_{i=1}^H \Mb_{t+1}^{[i-1]} \wb_{t+1-i},\\
        &f_{i,t}(\Mb_{t-H:t}|\Ab,\Bb,\{\wb\})\\ = &c_{i,t}\big(\yb^{\Kb}_{t+1}(\Mb_{t-H:t}|\Ab,\Bb,\{\wb\}), \vb^{\Kb}_{t+1}(\Mb_{t-H:t}|\Ab,\Bb,\{\wb\})\big),\\
        &i\in[m].
    \end{align*}
    \end{minipage}
    }
\end{equation*}
\end{definition}
For $f_{i,t}(\Mb_{t-H:t}|\Ab,\Bb,\{\wb\})$, if a time-invariant policy $(\Mb,\Kb)$ is applied over time, we denote it as $f_{i,t}(\Mb|\Ab,\Bb,\{\wb\})$.

\subsection{Distributed OCO with Memory}
A key to the analysis of distributed online control problems is to connect them to the framework of distributed online convex optimization (OCO) with memory. In centralized OCO with memory \cite{anava2015online}, at round $t$ the online learner picks an action $\xb_t\in\mathcal{X}\subset \mathrm{R}^d$, then the cost function $f_t:\mathcal{X}^{H+1}\to \mathrm{R}$ is revealed to the learner that incurs the cost $f_t(\xb_{t-H},\ldots,\xb_t)$. The goal is to minimize the {\it policy regret} \cite{arora2012online} defined as 
\begin{equation*}
    \sum_{t=H+1}^T f_t(\xb_{t-H},\ldots,\xb_t) - \min_{\xb\in\mathcal{X}}\sum_{t=H+1}^t f_t(\xb,\ldots,\xb).
\end{equation*}
It was shown that for convex functions, the upper bound of the policy regret is $O(\sqrt{T})$ \cite{anava2015online}. In this work, we will tackle the distributed online control problem using the distributed variant of OCO with memory, where the individual policy regret of agent $j$ is defined as 
\begin{equation*}
    \sum_{t=H+1}^T\sum_{i=1}^m f_{i,t}(\xb_{j,t-H},\ldots,\xb_{j,t}) - \min_{\xb\in\mathcal{X}}\sum_{t=H+1}^t f_t(\xb,\ldots,\xb),
\end{equation*}
where $f_t = \sum_i f_{i,t}$. For the distributed case, the goal is to make the decision sequence $(\xb_{j,1},\ldots,\xb_{j,T})$ of agent $j$ competitive to the best centralized decision in hindsight under the situation that each agent $j\in [m]$ has access only to its local functions $f_{j,t}$, which necessitates the communication among agents.

\section{Distributed Online Control with Known System Parameters}
In this section, we propose D-DFC, a distributed variant of DFC that generates controllers only with {\it local} information and {\it local} communication. At round $t$, the network performs a distributed online gradient step over the cost functions defined by the surrogate states and actions. All agents first share their current iterates $\{\Mb_{i,t}\}$ with their neighbors based on the network topology and update the iterates using local gradients $\nabla_{\Mb} f_{i,t}(\Mb_{i,t}|\Ab,\Bb,\{\wb\})$. Then, the updated policy parameters $\Mb_{i,t+1}$  will be used to decide the next local action $\ub_{i,t}$ for agent $i$. This process is summarized in Algorithm \ref{alg:D-online control}. There are a number of hyper-parameters, such as $\eta$ and $H$, which must be chosen as functions of time horizon $T$. Under suitable choices of these hyper-parameters, the individual regret can be guaranteed to have an upper bound that is sub-linear in $T$.

\begin{algorithm}[h]
\caption{Distributed Disturbance Feedback Controller (Known System)}
\label{alg:D-online control}
\begin{algorithmic}[1]
    \STATE {\bfseries Require:} number of agents $m$, doubly stochastic matrix $\Pb\in \mathrm{R}^{m\times m}$, parameters $\kappa, \gamma,\eta, H$, time horizon $T$.

    \STATE  Define $\mathcal{M}=\Big\{\Mb=\{\Mb^{[0]},\ldots,\Mb^{[H-1]}\}:\norm{\Mb^{[i]}}\leq 2\kappa^3(1-\gamma)^i,~\text{for}~i=0,\ldots,H-1\Big\}$.
   
    \STATE {\bf Initialize:} $\forall i\in[m]$, randomly generate the same $\Mb_{i,1}\in\mathcal{M}$ for all $i$. 
  
    \FOR{$t=1,2,\ldots,T$}
        \FOR{$i=1,2,\ldots,m$}
            \STATE Determine the control $\ub^{\Kb}_{i,t}$:
            \begin{equation*}
                \ub^{\Kb}_{i,t} = -\Kb\xb^{\Kb}_{i,t} + \sum_{j=1}^H \Mb_{i,t}^{[j-1]}\wb_{t-j}.
            \end{equation*}
            \STATE Observe $\xb^{\Kb}_{i,t+1}$ and compute $\wb_{t}=\Ab\xb^{\Kb}_{i,t} - \Bb\ub^{\Kb}_{i,t}$.                
            
            \STATE Compute $\widehat{\Mb}_{i,t+1} = \sum_{j=1}^m [\Pb]_{ji}\Mb_{j,t}-\eta\nabla_{\Mb} f_{i,t}(\Mb_{i,t}|\Ab,\Bb,\{\wb\})$.

            \STATE Compute $\Mb_{i,t+1} = \Pi_{\mathcal{M}}(\widehat{\Mb}_{i,t+1})$.
            
        \ENDFOR
   \ENDFOR
\end{algorithmic}
\end{algorithm}

\begin{remark}
Line 8 of Algorithm \ref{alg:D-online control} involves the communication between agents. We reiterate that though the linear combination of iterates is summed from $1$ to $m$, only the information from actual neighbors is taken into consideration as $[\Pb]_{ji} = 0$ if agents $i$ and $j$ are not neighbors. As such, the algorithm indeed runs with local information and in a distributed fashion.
\end{remark}

We adhere to the following standard assumptions for the analysis of online control \cite{agarwal2019online}\footnote{Without loss of generality, we assume constants $\kappa,W,G_c > 1$ as they are upper bounds.}:
\begin{assumption}\label{A1}
    The matrices governing the dynamics are bounded, i.e., $\norm{\Ab},\norm{\Bb}\leq \kappa$. The noise $\wb_{t}$ evolves adversarially and is bounded, i.e., $\norm{\wb_{t}}\leq W$.
\end{assumption}
\begin{assumption}\label{A2}
    The cost $c_{i,t}(\xb,\ub)$ is convex. Furthermore, as long as $\norm{\xb},\norm{\ub}\leq D$, then $\nabla_{(\xb,\ub)} c_{i,t}(\xb,\ub)\leq G_c D$,\; where $G_c$ is a constant.
\end{assumption}
\begin{assumption}\label{A3}
    All agents in the network have the common knowledge of a linear controller $\Kb$, which is $(\kappa,\gamma)$-strongly stable with respect to the underlying system $(\Ab,\Bb)$.
\end{assumption}
The assumption on the knowledge of a common linear controller $\Kb$ is mainly for the ease of analysis. If  agents have access to different controllers, the result remains the same in terms of regret order (as long as all those controllers are strongly stable). Furthermore, this assumption is commonly adopted in the literature as when the system parameters $(\Ab,\Bb)$ are known, we can extract a stabilizing controller from a SDP formulation \cite{cohen2018online,chen2021black}. Also, for the unknown dynamics case, a stabilizing controller can be efficiently computed  from the system estimates without affecting the order of the final regret bound \cite{chen2021black, lale2022reinforcement}.

We now present the main result of this section in the following theorem.
\begin{theorem}\label{T: Regret bound (known case)} Suppose Assumptions \ref{A1}, \ref{A2}, and \ref{A3} hold. By running Algorithm \ref{alg:D-online control} with $\eta = \Theta(\frac{1}{\sqrt{T}})$ and $H = \Theta(\log T)$, we have that the individual regret of any agent $j\in [m]$ is upper bounded as follows,
\begin{equation*}
    J_T^j(\mathcal{A}_1) - \min_{\Kb\in\mathcal{K}}J_T(\Kb) = O(\frac{\sqrt{T}\log T}{(1-\beta)^2}),
\end{equation*}
\end{theorem}
where $\mathcal{K}$ denotes the class of $(\kappa,\gamma)$-strongly stable controllers.

From Theorem \ref{T: Regret bound (known case)}, we can see that the regret bound of distributed DFC recovers that of the centralized DFC in \cite{agarwal2019online} in terms of the order of $T$. Intuitively, two main components contribute to the regret analysis: the approximation part, where the actual states (actions) are approximated by the surrogate states (actions), and the regret part as a result of distributed OCO with memory. For the first part, it is shown that the approximation gap shrinks geometrically in terms of $H$, which can be mitigated in the final regret bound when we choose a parameter $H$ that scales at least logarithmically with respect to time horizon $T$. The second (dominant) part via the analysis of D-OGD is sub-linear and can be handled optimally when we set $\eta=\Theta(1/\sqrt{T})$. Furthermore, the effect of the network topology is  characterized by the dependence on $\beta$, the second largest singular value of $\Pb$. As we can see in the regret bound, when the network is more connected (i.e., $\beta$ is smaller), the resulting regret bound is tighter.

\section{Distributed Online Control with Unknown System Parameters}
In this section, we consider the case where the system parameters are unknown. To ensure a sub-linear regret bound for this scenario, we take an explore-then-commit approach such that in the first $T_0$ iterations of Algorithm \ref{alg:D-online control (unknown system)}, all agents jointly estimate the system parameters in a distributed fashion using Algorithm \ref{alg:parameter estimation}, which is a distributed variant of the system identification proposed in \cite{hazan2020nonstochastic}. 
In particular, it can be shown that by intentionally adding an i.i.d. noise sequence to a linear controller, we can estimate and extract a quantity that is related to system parameters (using line 8 in Algorithm \ref{alg:parameter estimation}). We further perform the communication step in line 10 to improve the statistical efficiency.

After the estimation phase, at every round each agent estimates the current noise using its own system estimate (line 9) and applies  D-DFC based on its past estimated noises. As opposed to Algorithm \ref{alg:D-online control}, the distributed gradient step of each agent is applied over the cost defined with respect to different parameters since system and noise estimates are different across agents. Evidently, the summation of these local surrogate costs $\sum_i f_{i,t}(\cdot|\widehat{\Ab}_i, \widehat{\Bb}_i, \{\hat{\wb}_i\})$ is different to the centralized cost of $f_t(\cdot) = \sum_i f_{i,t}(\cdot|\Ab,\Bb,\{\wb\})$. Therefore, the distributed problem in the unknown dynamics case does not directly fit into the framework of OCO with memory as the centralized case (e.g., \cite{agarwal2019online,hazan2020nonstochastic}). To tackle this issue, we provide the analysis that quantifies the effect of different parameters on the evolution of the (surrogate) states. Finally, with suitable choices for $\eta$ and $H$, a sub-linear regret bound is guaranteed for the individual regret.

\begin{algorithm}[h]
\caption{System Identification}
\label{alg:parameter estimation}
\begin{algorithmic}[1]
    \STATE {\bfseries Require:} number of agents $m$, doubly stochastic matrix $\Pb\in \mathrm{R}^{m\times m}$, time horizons $T_{collect}$ and $T_{exchange}$.
  
    \FOR{$t=1,2,\ldots,T_{collect}$}
        \FOR{$i=1,2,\ldots,m$}
            \STATE Record the observed state $\xb_{i,t}$.
            
            \STATE Determine the control $\ub_{i,t}$:
            \begin{equation*}
                \ub_{i,t} = -\Kb\xb_{i,t} + \xi_{i,t} \text{where } \xi_{i,t}\sim_{i.i.d} \{\pm 1\}^{d_2}.
            \end{equation*}   
        \ENDFOR
   \ENDFOR

    \STATE Define 
    \begin{equation*}
    \begin{split}
        &\Nb^{T_{collect}}_{i,k-1} = \frac{1}{T_{collect} - (q+1)}\sum_{t=1}^{T_{collect}-(q+1)}\xb_{i,t+k}\xi^{\top}_{i,t},\\
        &\forall i\in[m] \text{ and } k\in[q+1].
    \end{split}
    \end{equation*}

    \FOR{$t=T_{collect}+1,\ldots,T_{collect} + T_{exchange}$}
    \STATE
        \begin{equation*}
            \Nb^{t}_{i,k-1} = \sum_{j=1}^m [\Pb]_{ji}\Nb^{t-1}_{j,k-1},\;\forall i\in[m] \text{ and } k\in[q+1].
        \end{equation*}    
    \ENDFOR

    \STATE let $T_0 = T_{collect} + T_{exchange}$.

    \STATE Define $\Cb_{i,0} = \left[\Nb^{T_0}_{i,0},\ldots,\Nb^{T_0}_{i,q-1}\right]$ and $\Cb_{i,1} = \left[\Nb^{T_0}_{i,1},\ldots,\Nb^{T_0}_{i,q}\right],\;\forall i\in[m]$, and return $\widehat{\Ab}_i$ and $\widehat{\Bb}_i$ as 
    \begin{equation*}
    \resizebox{\linewidth}{!}
    {
    \begin{minipage}{\linewidth}
    \begin{align*}
        \widehat{\Bb}_i = \Nb^{T_0}_{i,0},\; \widehat{\Ab}^{\prime}_i = \Cb_{i,1}\Cb_{i,0}^{\top}(\Cb_{i,0}\Cb_{i,0}^{\top})^{-1},\; \widehat{\Ab}_i = \widehat{\Ab}^{\prime}_i + \widehat{\Bb}_i\Kb. 
    \end{align*}
    \end{minipage}
    }
    \end{equation*}

\end{algorithmic}
\end{algorithm}
\begin{algorithm}[h]
\caption{Distributed Disturbance Feedback Controller (Unknown System)}
\label{alg:D-online control (unknown system)}
\begin{algorithmic}[1]
    \STATE {\bfseries Require:} number of agents $m$, doubly stochastic matrix $\Pb\in \mathrm{R}^{m\times m}$, parameters $\kappa, \gamma,\eta, H$, time horizon $T$.
    \STATE Define $\mathcal{M}=\Big\{\Mb=\{\Mb^{[0]},\ldots,\Mb^{[H-1]}\}:\norm{\Mb^{[i]}}\leq 2\kappa^3(1-\gamma)^i,~\text{for}~i=0,\ldots,H-1\Big\}$.

    \STATE Specify $T_{collect}$ and $T_{exchange}$ based on given hyper-parameters and run Algorithm \ref{alg:parameter estimation} to learn the system parameters $\{\widehat{\Ab}_i, \widehat{\Bb}_i\}_{i\in[m]}$ in a distributed fashion.

    \STATE let $T_0 = T_{collect} + T_{exchange}$.
    \STATE {\bf Initialize:} $\forall i\in[m]$, randomly choose the same $\Mb_{i,T_0+1}\in\mathcal{M}$ for all $i\in[m]$ and define $\hat{\wb}_{i,t}=0$ for $t\leq T_0$.
  
    \FOR{$t=T_0+1,\ldots,T$}
        \FOR{$i=1,2,\ldots,m$}
            \STATE Determine the control $\ub_{i,t}$:
            \begin{equation*}
                \ub_{i,t} = -\Kb\xb_{i,t} + \sum_{j=1}^H \Mb_{i,t}^{[j-1]}\hat{\wb}_{i,t-j}.
            \end{equation*}
            \STATE Observe $\xb_{i,t+1}$ and compute $\hat{\wb}_{i,t}=\widehat{\Ab}_i\xb_{i,t} - \widehat{\Bb}_i\ub_{i,t}$.                
            \STATE Compute $\widehat{\Mb}_{i,t+1} = \sum_{j=1}^m [\Pb]_{ji}\Mb_{j,t}-\eta\nabla f_{i,t}(\Mb_{i,t}|\widehat{\Ab}_i, \widehat{\Bb}_i, \{\hat{\wb}_i\})$.

            \STATE Compute $\Mb_{i,t+1} = \Pi_{\mathcal{M}}(\widehat{\Mb}_{i,t+1})$.
            
        \ENDFOR
   \ENDFOR
\end{algorithmic}
\end{algorithm}

\begin{assumption}\label{A4}
    The linear system $(\Ab - \Bb\Kb, \Bb)$ is $(q,\kappa)$-strongly controllable.
\end{assumption}

\begin{theorem}\label{T: Regret bound (unknown case)} Suppose Assumptions \ref{A1}, \ref{A2}, \ref{A3} and \ref{A4} hold. By running Algorithm \ref{alg:D-online control (unknown system)} with $\eta = \Theta(\frac{1}{\sqrt{T}})$, $H = \Theta(\log T)$, $T_{collect} = \Theta(T^{2/3})$, $T_{exchange} = \Theta(\log T)$, we have with probability at least $(1-\delta)$ that the individual regret of any agent $j\in [m]$ is upper bounded as follows,
\begin{equation*}
 \begin{split}
     &J_T^j(\mathcal{A}_3) - \min_{\Kb\in\mathcal{K}}J_T(\Kb)\\ = &O\Big(\big(T^{2/3}\sqrt{\log \delta^{-1}} + \frac{\sqrt{T}}{(1-\beta)^2}\big)\text{poly}(\log T)\Big),
 \end{split}
\end{equation*}
\end{theorem}
where $\mathcal{K}$ denotes the class of $(\kappa,\gamma)$-strongly stable controllers.

Compared to Theorem \ref{T: Regret bound (known case)}, the increased regret order is due to the trade-off between exploration and exploitation, where the optimal rate is achieved when $T_{collect}$ is set to be $\Theta(T^{2/3})$. On the other hand, the above regret bound recovers the result of centralized case with unknown systems \cite{hazan2020nonstochastic} up to log factors. As mentioned earlier, the analysis of D-OL with memory is not directly applicable due to the fact that the local surrogate function is built on individual agent estimates, and this gap results in the extra poly-log factors in the final bound. The detailed derivation is provided in the supplementary material, but we provide a proof sketch here. To bound the regret, we need to quantify the contributions of three main components: {\bf (i) Approximation of the actual state (action) with the surrogate state (action):} This gap is bounded using the strong stability of $\Kb$. Similar analysis is also applied for the approximation of the comparator. {\bf (ii) The impact of surrogate functions built on individual estimates:} In the supplementary (Lemma 17), we show that the function gap can be bounded by the order of the estimate difference across agents (up to log factors). {\bf (iii) Distributed OCO with memory:} After handling the first two parts, the control problem can be transformed into D-OCO with memory, where we extend the analysis of the centralized case \cite{anava2015online} to the distributed setup and derive the final regret bound.


\section*{Conclusion}
In this work, we considered the distributed online control problem with potentially unknown LTI systems and time-varying convex costs. For the known dynamics case, we proposed an algorithm based on distributed DFC, which transforms the problem into a distributed OCO with memory, and the resulting regret bound is $O(\sqrt{T}\log T)$. For the unknown dynamics case, we provided a decentralized two-phase algorithm, which estimates the underlying dynamics and then applies the first algorithm using the estimates obtained from the system identification. We established a regret bound of $O(T^{2/3}\text{poly}(\log T))$ for this setup. For future work, possible extensions include the analysis of the {\it dynamic} performance criterion (i.e., when the comparator sequence is time-varying) and investigating distributed partially observable LTI systems.


\bibliographystyle{IEEEtran}
\bibliography{references}

\newpage
\onecolumn
\section{Supplementary}
\subsection{Approximation based on DFC}
As mentioned in the main context, the reparameterization of applying DFC with a finite length of past noises is the key step for fitting the online control problem into the online learning framework. In this section, we provide quantified approximation errors when the derived state (action) is approximated by the surrogate state (action) (Lemma \ref{L: Upper bounds of (surrogate)states and (surrogate)actions and the difference}) and when the comparator of a linear policy is approximated by a time-invariant DFC (Lemma \ref{L: Approximate a linear policy using time-invariant DFC policy}). 

\begin{lemma}\label{L: Upper bound of the disturbance-state transfer matrix}(Upper bound of the disturbance-state transfer matrix) Let $\Kb$ be $(\kappa',\gamma')$ strongly stable w.r.t. the system $(\Ab^{\prime},\Bb^{\prime})$ where $\norm{\Bb^{\prime}}\leq \kappa_{B^{\prime}}$. Suppose $\{\Mb_t\}$ is a sequence of parameters such that $\Mb_t\in\mathcal{M}=\{\Mb:\Mb^{[i]}\leq C(1-\gamma)^i, \gamma\geq\gamma'\},\;\forall t$, then we have $\forall i,t$,
    \begin{equation*}
        \norm{\Psi^{\Kb,h}_{t,i}(\Mb_{t-h:t}|\Ab^{\prime},\Bb^{\prime})} \leq (\kappa')^2(1-\gamma')^i \mathbb{1}_{i\leq h} + H \kappa_{B^{\prime}}(\kappa')^2C(1-\gamma')^{i-1}.
    \end{equation*}
    \end{lemma}
    \begin{proof}
    \begin{equation*}
    \begin{split}
        \norm{\Psi^{\Kb,h}_{t,i}(\Mb_{t-h:t}|\Ab^{\prime},\Bb^{\prime})} &= \norm{(\Ab^{\prime} - \Bb^{\prime}\Kb)^i \mathbb{1}_{i\leq h} + \sum_{j=0}^h (\Ab^{\prime} - \Bb^{\prime}\Kb)^j\Bb^{\prime}\Mb^{[i-j-1]}_{t-j}\mathbb{1}_{i-j\in [1,H]}}\\
        &\leq (\kappa')^2(1-\gamma')^i \mathbb{1}_{i\leq h} + \sum_{j=0}^h (\kappa')^2(1-\gamma')^j\kappa_{B^{\prime}} C(1-\gamma)^{i-j-1}\mathbb{1}_{i-j\in [1,H]}\\
        &\leq (\kappa')^2(1-\gamma')^i \mathbb{1}_{i\leq h} + H \kappa_{B^{\prime}}(\kappa')^2C(1-\gamma')^{i-1}.
    \end{split}
    \end{equation*}
    \end{proof}
    
    \begin{lemma}\label{L: Upper bounds of (surrogate)states and (surrogate)actions and the difference}(Upper bounds of (surrogate)states and (surrogate)actions and the difference) 
    Let $\Kb$ be $(\kappa',\gamma')$ strongly stable w.r.t. the system $(\Ab^{\prime},\Bb^{\prime})$ where $\norm{\Bb^{\prime}}\leq \kappa_{B^{\prime}}$. For a linear system $\{\xb\}$ (assume $\xb_1=0$) evolving based on $(\Ab^{\prime}, \Bb^{\prime}, \{\wb^{\prime}\})$ where $\wb^{\prime}_{t}\leq W^{\prime}$. Suppose the control is chosen based on DFC with $\Mb_t\in\mathcal{M}=\{\Mb:\Mb^{[i]}\leq C(1-\gamma)^i, \gamma\geq\gamma'\},\;\forall t$, then we have for all $t$,
    \begin{equation*}
        \norm{\xb_t}, \norm{\yb^{\Kb}_t}, \norm{\ub^{\Kb}_t}, \norm{\vb^{\Kb}_t} \leq D        
    \end{equation*}
    \begin{equation*}
        \norm{\xb_t - \yb^{\Kb}_t} \leq (\kappa')^2(1-\gamma')^{H+1}D
    \end{equation*}
    \begin{equation*}
        \norm{\ub^{\Kb}_t - \vb^{\Kb}_t} \leq (\kappa')^3(1-\gamma')^{H+1}D,  
    \end{equation*}
    where $D = \frac{W^{\prime}\big((\kappa')^2 + H\kappa_{B^{\prime}}(\kappa')^2C\big)}{\gamma'\big(1-(\kappa')^2(1-\gamma')^{(H+1)}\big)} + \frac{W^{\prime}C}{\gamma}$.
    \end{lemma}
    \begin{proof}
    Based on the expression of DFC, we have
    \begin{equation*}
    \begin{split}
        \norm{\xb_{t+1}} &= \norm{(\Ab^{\prime} - \Bb^{\prime}\Kb)^{H+1}\xb_{t-H} + \sum_{i=0}^{2H} \Psi^{\Kb,H}_{t,i}(\Mb_{t-H:t}|\Ab^{\prime},\Bb^{\prime}) \wb^{\prime}_{t-i}}\\
        &\leq (\kappa')^2(1-\gamma')^{(H+1)}\norm{\xb_{t-H}} + W^{\prime}\sum_{i=0}^{2H} \norm{\Psi^{\Kb,H}_{t,i}(\Mb_{t-H:t}|\Ab^{\prime},\Bb^{\prime})}\\
        &\leq (\kappa')^2(1-\gamma')^{(H+1)}\norm{\xb_{t-H}} + W^{\prime}\left(\frac{(\kappa')^2 + H\kappa_{B^{\prime}}(\kappa')^2C}{\gamma'}\right),
    \end{split}    
    \end{equation*}
    which by expanding the recursion implies the following
    \begin{equation*}
        \norm{\xb_{t+1}} \leq \frac{W^{\prime}\big((\kappa')^2 + H\kappa_{B^{\prime}}(\kappa')^2C\big)}{\gamma'\big(1-(\kappa')^2(1-\gamma')^{(H+1)}\big)}\leq D.
    \end{equation*}
    For the surrogate state, we have
    \begin{equation*}
    \begin{split}
        \norm{\yb^{\Kb}_{t+1}(\Mb_{t-H:t}|\Ab^{\prime},\Bb^{\prime},\{\wb^{\prime}\})} &= \norm{\sum_{i=0}^{2H} \Psi^{\Kb,H}_{t,i}(\Mb_{t-H:t}|\Ab^{\prime},\Bb^{\prime}) \wb^{\prime}_{t-i}}\\
        &\leq W^{\prime}\sum_{i=0}^{2H} \norm{\Psi^{\Kb,H}_{t,i}(\Mb_{t-H:t}|\Ab^{\prime},\Bb^{\prime})}\\
        &\leq W^{\prime}\left(\frac{(\kappa')^2 + H\kappa_{B^{\prime}}(\kappa')^2C}{\gamma'}\right)\leq D.
    \end{split}    
    \end{equation*}
    And for the difference between the actual and surrogate states, we have
    \begin{equation*}
        \norm{\xb_{t+1} - \yb^{\Kb}_{t+1}(\Mb_{t-H:t}|\Ab^{\prime},\Bb^{\prime},\{\wb^{\prime}\})} \leq \norm{(\Ab^{\prime} - \Bb^{\prime}\Kb)^{H+1}}\norm{\xb_{t-H}}\leq (\kappa')^2(1-\gamma')^{(H+1)}D.
    \end{equation*}
    \end{proof}
    For the actual and surrogate actions, we have the following bounds
    \begin{equation*}
        \norm{\ub^{\Kb}_{t+1}}\leq \norm{\Kb}\norm{\xb_{t+1}} + W^{\prime}\sum_{i=1}^H \norm{\Mb_{t+1}^{[i-1]}}\leq \kappa'\norm{\xb_{t+1}} + \frac{W^{\prime}C}{\gamma}\leq D,
    \end{equation*}
    \begin{equation*}
        \norm{\vb^{\Kb}_{t+1}}\leq \norm{\Kb}\norm{\yb^{\Kb}_{t+1}} + W^{\prime}\sum_{i=1}^H \norm{\Mb_{t+1}^{[i-1]}}\leq \kappa'\norm{\yb^{\Kb }_{t+1}} + \frac{W^{\prime}C}{\gamma}\leq D.
    \end{equation*}
    And the difference is upper bounded as follows
    \begin{equation*}
        \norm{\ub^{\Kb}_{t+1} - \vb^{\Kb}_{t+1}}\leq \norm{\Kb}\norm{\xb^{\Kb}_{t+1}-\yb^{\Kb}_{t+1}}\leq (\kappa')^3(1-\gamma')^{(H+1)}D.
    \end{equation*}

    \begin{lemma}\label{L: Approximate a linear policy using time-invariant DFC policy}(Approximate a linear policy using time-invariant DFC policy) For any two $(\kappa, \gamma)$-strongly stable matrices $\Kb$ and $\Kb^*$, there is a policy $\pi(\Mb_*, \Kb)$ where $\Mb_*=(\Mb_*^{[0]},\ldots,\Mb_*^{[H-1]})$ defined as 
    \begin{equation*}
        \Mb_*^{[i]} = (\Kb - \Kb^*)(\Ab - \Bb\Kb^*)^{i}
    \end{equation*}
    such that the difference of the accumulated costs of the time-invariant DFC $(\Mb_*,\Kb)$ and the linear policy $\Kb^*$ is upper bounded as follows
    \begin{equation*}
        \sum_{t=1}^T \left[c_t\big(\xb^{\Kb}_t(\Mb_*), \ub^{\Kb}_t(\Mb_*)\big) - c_t\big(\xb^{\Kb^*}_t(0), \ub^{\Kb^*}_t(0)\big)\right]\leq TmG_c D\frac{6WH \kappa_B^2\kappa^6(1-\gamma)^{H-1}}{\gamma},
    \end{equation*}
    where 
    \begin{equation*}
    \begin{split}
        \xb^{\Kb}_{t+1}(\Mb_*) &= \Ab\xb^{\Kb}_t(\Mb_*) + \Bb\ub^{\Kb}_t(\Mb_*) + \wb_t,\\
        \ub^{\Kb}_t(\Mb_*) &= -\Kb\xb^{\Kb}_t(\Mb_*) + \sum_{k=1}^H \Mb_*^{[k-1]}\wb_{t-k}. 
    \end{split}
    \end{equation*}
    \end{lemma}
    \begin{proof}
    Expanding the linear and DFC policies recursively, we have
    \begin{equation}\label{Eq1: Approximate a linear policy using time-invariant DFC policy}
    \begin{split}
        \xb^{\Kb^*}_{t+1}(0) &= \sum_{i=0}^{t-1} \Tilde{\Ab}^i_{\Kb^*}\wb_{t-i},\\
        \xb^{\Kb}_{t+1}(\Mb_*) &= \sum_{i=0}^{t-1} \Psi^{\Kb,t-1}_{t,i}(\Mb_*) \wb_{t-i}.
    \end{split}
    \end{equation}
    Also, based on the definition of the disturbance-state transfer matrix of DFC, we have $\forall i\leq H$ (we note that based on Equation \eqref{Eq1: Approximate a linear policy using time-invariant DFC policy}, $i$ is always less than $t-1$),
    \begin{equation}\label{Eq2: Approximate a linear policy using time-invariant DFC policy}
    \begin{split}
        \Psi^{\Kb,t-1}_{t,i}(\Mb_*) &= \Tilde{\Ab}_{\Kb}^i \mathbb{1}_{i\leq t-1} + \sum_{j=0}^{t-1} \Tilde{\Ab}_{\Kb}^j\Bb\Mb^{[i-j-1]}_{*}\mathbb{1}_{i-j\in [1,H]}\\
        &= \Tilde{\Ab}_{\Kb}^i + \sum_{j=0}^{i-1} \Tilde{\Ab}_{\Kb}^j\Bb\Mb^{[i-j-1]}_{*}\\
        &= \Tilde{\Ab}_{\Kb}^i + \sum_{j=0}^{i-1} \Tilde{\Ab}_{\Kb}^j\Bb (\Kb - \Kb^*)\Tilde{\Ab}_{\Kb^*}^{i-j-1}\\
        &= \Tilde{\Ab}_{\Kb}^i + \sum_{j=0}^{i-1} \Tilde{\Ab}_{\Kb}^j(\Tilde{\Ab}_{\Kb^*} - \Tilde{\Ab}_{\Kb})\Tilde{\Ab}_{\Kb^*}^{i-j-1}\\
        &= \Tilde{\Ab}_{\Kb}^i + \sum_{j=0}^{i-1} \Tilde{\Ab}_{\Kb}^j\Tilde{\Ab}_{\Kb^*}^{i-j} - \Tilde{\Ab}_{\Kb}^{j+1}\Tilde{\Ab}_{\Kb^*}^{i-j-1}\\
        &= \Tilde{\Ab}_{\Kb^*}^i.
    \end{split}
    \end{equation}
    Based on Equations \eqref{Eq1: Approximate a linear policy using time-invariant DFC policy} and \eqref{Eq2: Approximate a linear policy using time-invariant DFC policy}, we get
    \begin{equation}\label{Eq3: Approximate a linear policy using time-invariant DFC policy}
    \begin{split}
        \norm{\xb^{\Kb}_{t+1}(\Mb_*) - \xb^{\Kb^*}_{t+1}(0)} &= \norm{\sum_{i=H+1}^{t-1} \big(\Psi^{\Kb,t-1}_{t,i}(\Mb_*) - \Tilde{\Ab}^i_{\Kb^*}\big) \wb_{t-i}}\\
        &\leq W\sum_{i=H+1}^{t-1}\left[\norm{\Psi^{\Kb,t-1}_{t,i}(\Mb_*)} + \norm{\Tilde{\Ab}^i_{\Kb^*}}\right]\\
        &\leq W\sum_{i=H+1}^{t-1}\left[\kappa^2(1-\gamma)^i + H \kappa_B^2\kappa^5(1-\gamma)^{i-1} + \kappa^2(1-\gamma)^i\right]\\
        &\leq \frac{3WH \kappa_B^2\kappa^5(1-\gamma)^H}{\gamma},
    \end{split}    
    \end{equation}
    where the second inequality is based on Lemma \ref{L: Upper bound of the disturbance-state transfer matrix} and the fact that $\norm{\Mb_*^{[i]}}\leq \kappa_B\kappa^3(1-\gamma)^i$.
    Also based on the policies, we have
    \begin{equation}\label{Eq4: Approximate a linear policy using time-invariant DFC policy}
    \begin{split}
        \norm{\ub^{\Kb}_{t+1}(\Mb_*) - \ub^{\Kb^*}_{t+1}(0)} &= \norm{-\Kb\xb^{\Kb}_{t+1}(\Mb_*) + \Kb^*\xb^{\Kb^*}_{t+1}(0) +\sum_{k=0}^{H-1} \Mb_*^{[k]}\wb_{t-k} }\\
        &\leq \sum_{i=H}^{t-1} \norm{\big(-\Kb\Psi^{\Kb,t-1}_{t,i}(\Mb_*) +  \Kb^*\Tilde{\Ab}^i_{\Kb^*}\big) \wb_{t-i}}\\
        &\leq W\sum_{i=H}^{t-1}\left[\kappa^3(1-\gamma)^i + H \kappa_B^2\kappa^6(1-\gamma)^{i-1} + \kappa^3(1-\gamma)^i\right]\\
        &\leq \frac{3WH \kappa_B^2\kappa^6(1-\gamma)^{H-1}}{\gamma}.
    \end{split}    
    \end{equation}
    Then based on Assumption \ref{A2}, Lemma \ref{L: Upper bounds of (surrogate)states and (surrogate)actions and the difference}, Equations \eqref{Eq3: Approximate a linear policy using time-invariant DFC policy} and \eqref{Eq4: Approximate a linear policy using time-invariant DFC policy}, we get
    \begin{equation}\label{Eq5: Approximate a linear policy using time-invariant DFC policy}
    \begin{split}
        &c_{t+1}\big(\xb^{\Kb}_{t+1}(\Mb_*), \ub^{\Kb}_{t+1}(\Mb_*)\big) - c_{t+1}\big(\xb^{\Kb^*}_{t+1}(0), \ub^{\Kb^*}_{t+1}(0)\big)\\
        \leq &mG_c D\left(\norm{\xb^{\Kb}_{t+1}(\Mb_*) - \xb^{\Kb^*}_{t+1}(0)} + \norm{\ub^{\Kb}_{t+1}(\Mb_*) - \ub^{\Kb^*}_{t+1}(0)}\right)\\
        \leq &mG_c D\frac{6WH \kappa_B^2\kappa^6(1-\gamma)^{H-1}}{\gamma}.
    \end{split}
    \end{equation}
    Summing the inequality above over time, the final result is derived.
    
    \end{proof}

    \subsection{Bounds for The Distributed Online Learning with Memory and the Individual Regret (Known System)}
    In this section, we discuss the individual regret bound for systems with known parameters via the analysis of the distributed online learning with memory based on the parameterization of DFC. Specifically speaking, Lemma \ref{L: Lipschitz constant in terms of the framework of OCO with memory} characterizes the Lipschitz continuity of the function parameterized by DFC, Lemma \ref{L: Gradient bound for reparameterized functions} provides the gradient bound of the reparameterized function and Lemma \ref{L: Consecutive distance} quantifies the difference between iterates learned by different agents when D-OGD is applied based on the network topology. With these auxiliary lemmas, Theorem \ref{T: Regret bound for OCO with memory} shows the regret bound for the transformed distributed OL with memory, which is later used to prove Theorem \ref{T: Regret bound (known case)}.
    
    \begin{lemma}\label{L: Lipschitz constant in terms of the framework of OCO with memory} (Lipschitz constant in terms of the framework of OCO with memory) Let $\Kb$ be $(\kappa',\gamma')$ strongly stable w.r.t. the system $(\Ab^{\prime},\Bb^{\prime})$ where $\norm{\Bb^{\prime}}\leq \kappa_{B^{\prime}}$. Consider the same noise sequence $\{\wb^{\prime}\}$ ($\norm{\wb^{\prime}_t}\leq W'$ for all $t$) and two policy sequences $\{\Mb_{t-H},\ldots, \Mb_{t-k},\ldots, \Mb_{t+1}\}$ and $\{\Mb_{t-H},\ldots, \Tilde{\Mb}_{t-k},\ldots, \Mb_{t+1}\}$ which differ at a time step $t-k$ for $k\in\{-1,\ldots,H\}$, then under Assumption \ref{A2} we have
    \begin{equation*}
    \resizebox{\linewidth}{!}
    {
    \begin{minipage}{\linewidth}
    \begin{align*}
        &|f_{i,t+1}(\Mb_{t-H},\ldots, \Mb_{t-k},\ldots, \Mb_{t}|\Ab^{\prime},\Bb^{\prime},\{\wb^{\prime}\}) - f_{i,t+1}(\Mb_{t-H},\ldots, \Tilde{\Mb}_{t-k},\ldots, \Mb_{t}|\Ab^{\prime},\Bb^{\prime},\{\wb^{\prime}\})|\\ 
        \leq &G_c D 3\kappa_{B^{\prime}}\kappa^3W^{\prime} \sum_{j=0}^{H-1}\norm{\Mb_{t-k}^{[j]} - \Tilde{\Mb}_{t-k}^{[j]}},
    \end{align*}
    \end{minipage}
    }
    \end{equation*}
    where $D = \frac{W^{\prime}\big((\kappa')^2 + H\kappa_{B^{\prime}}(\kappa')^2C\big)}{\gamma'\big(1-(\kappa')^2(1-\gamma')^{(H+1)}\big)} + \frac{W^{\prime}C}{\gamma}$.
    \end{lemma}
    \begin{proof}
        Based on the definition of the surrogate state and action, we have the following upper bounds\\ (Define $\yb^{\Kb}_{t+1}(\Mb_{t-H:t}|\Ab^{\prime},\Bb^{\prime},\{\wb^{\prime}\})$ as $\yb^{\Kb}_{t+1}$. Similar for $\Tilde{\yb}^{\Kb}_{t+1}$, $\vb^{\Kb}_{t+1}$ and $\Tilde{\vb}^{\Kb}_{t+1}$),
    \begin{equation*}
    \begin{split}
        \norm{\yb^{\Kb}_{t+1} - \Tilde{\yb}^{\Kb}_{t+1}} &= \norm{\sum_{j=0}^{2H}(\Ab^{\prime} - \Bb^{\prime}\Kb)^k\Bb'(\Mb_{t-k}^{[j-k-1]} - \Tilde{\Mb}_{t-k}^{[j-k-1]})\wb^{\prime}_{t-j}\mathbb{1}_{j-k\in[1,H]}}\\
        &\leq \kappa_{B^{\prime}}(\kappa')^2(1-\gamma')^kW^{\prime} \sum_{j=0}^{H-1}\norm{\Mb_{t-k}^{[j]} - \Tilde{\Mb}_{t-k}^{[j]}} 
    \end{split}
    \end{equation*}
    and    
    \begin{equation*}
    \resizebox{\linewidth}{!}
    {
    \begin{minipage}{\linewidth}
    \begin{align*}
        \norm{\vb^{\Kb}_{t+1} - \Tilde{\vb}^{\Kb}_{t+1}} &= \norm{-\Kb(\yb^{\Kb}_{t+1} - \Tilde{\yb}^{\Kb}_{t+1}) + \mathbb{1}_{k=-1}\sum_{j=1}^H (\Mb_{t+1}^{[j-1]}-\Tilde{\Mb}_{t+1}^{[j-1]})\wb^{\prime}_{t+1-j}}\\
        &\leq \kappa_{B^{\prime}}(\kappa')^3(1-\gamma')^kW^{\prime} \sum_{j=0}^{H-1}\norm{\Mb_{t-k}^{[j]} - \Tilde{\Mb}_{t-k}^{[j]}} + \norm{\mathbb{1}_{k=-1}\sum_{j=1}^H (\Mb_{t+1}^{[j-1]}-\Tilde{\Mb}_{t+1}^{[j-1]})\wb^{\prime}_{t+1-j}}\\
        &\leq 2\kappa_{B^{\prime}}(\kappa')^3W^{\prime} \sum_{j=0}^{H-1}\norm{\Mb_{t-k}^{[j]} - \Tilde{\Mb}_{t-k}^{[j]}}.
    \end{align*}
    \end{minipage}
    }
    \end{equation*}
    Based the expression above and the assumption on the function gradient, we have
    \begin{equation*}
    \begin{split}
        &|f_{i,t+1}(\Mb_{i,t-H},\ldots, \Mb_{i,t-k},\ldots, \Mb_{i,t}) - f_{i,t+1}(\Mb_{i,t-H},\ldots, \Tilde{\Mb}_{i,t-k},\ldots, \Mb_{i,t})|\\
    = &|f_{i,t+1}(\yb_{t+1}^{\Kb}, \vb_{t+1}^{\Kb}) - f_{i,t+1}(\Tilde{\yb}_{t+1}^{\Kb}, \Tilde{\vb}_{t+1}^{\Kb})|\\
    \leq &G_c D(\sqrt{\norm{\yb^{\Kb}_{t+1} - \Tilde{\yb}^{\Kb}_{t+1}}^2 + \norm{\vb^{\Kb}_{t+1} - \Tilde{\vb}^{\Kb}_{t+1}}^2})\\
    \leq &G_c D(\norm{\yb^{\Kb}_{t+1} - \Tilde{\yb}^{\Kb}_{t+1}} + \norm{\vb^{\Kb}_{t+1} - \Tilde{\vb}^{\Kb}_{t+1}})\\
    \leq &G_c D 3\kappa_{B^{\prime}}(\kappa')^3W^{\prime} \sum_{j=0}^{H-1}\norm{\Mb_{t-k}^{[j]} - \Tilde{\Mb}_{t-k}^{[j]}}.
    \end{split}    
    \end{equation*}
    
    \end{proof}

    \begin{lemma}\label{L: Gradient bound for reparameterized functions}(Gradient bound for $\nabla_{\Mb}f_{t+1}(\Mb)$)Let $\Kb$ be $(\kappa',\gamma')$ strongly stable w.r.t. the system $(\Ab^{\prime},\Bb^{\prime})$ where $\norm{\Bb^{\prime}}\leq \kappa_{B^{\prime}}$. For any $\Mb\in \mathcal{M}=\{\Mb: \Mb^{[i-1]}\leq C(1-\gamma)^i, \gamma\geq\gamma'\}$, the derivative of the ideal cost $f_{i,t}(\Mb|\Ab^{\prime}, \Bb^{\prime},\{\wb^{\prime}\})$ by applying the time-invariant $\Mb$ on the system $(\Ab^{\prime}, \Bb^{\prime},\{\wb^{\prime}\})$ is upper bounded as follows
    \begin{equation*}
        \norm{\nabla_{\Mb}f_{i,t+1}(\Mb|\Ab^{\prime},\Bb^{\prime},\{\wb^{\prime}\})}_F\leq G_c D \sqrt{d}d\sqrt{H} (\frac{(\kappa')^3\kappa_{B^{\prime}} W^{\prime}}{\gamma'} + W^{\prime}),
    \end{equation*}    
    where $D = \frac{W^{\prime}\big((\kappa')^2 + H\kappa_{B^{\prime}}(\kappa')^2C\big)}{\gamma'\big(1-(\kappa')^2(1-\gamma')^{(H+1)}\big)} + \frac{W^{\prime}C}{\gamma}$ and $d=\max\{d_1,d_2\}$.
    \end{lemma}
    \begin{proof}
        We first decompose the gradient using the chain rule as follows    
    \begin{equation*}
    \resizebox{\linewidth}{!}
    {
    \begin{minipage}{\linewidth}
    \begin{align*}
        &\norm{\nabla_{\Mb}f_{i,t+1}(\Mb|\Ab^{\prime},\Bb^{\prime},\{\wb^{\prime}\})}\\
        \leq &\sum_k\norm{\frac{\partial f_{i,t+1}(\Mb|\Ab^{\prime},\Bb^{\prime},\{\wb^{\prime}\})}{\partial[\yb_{t+1}^{\Kb}(\Mb|\Ab^{\prime},\Bb^{\prime},\{\wb^{\prime}\}), \vb_{t+1}^{\Kb}(\Mb)|\Ab^{\prime},\Bb^{\prime},\{\wb^{\prime}\}]_k}}\norm{\diff{[\yb_{t+1}^{\Kb}(\Mb|\Ab^{\prime},\Bb^{\prime},\{\wb^{\prime}\}), \vb_{t+1}^{\Kb}(\Mb|\Ab^{\prime},\Bb^{\prime},\{\wb^{\prime}\})]_k}{\Mb}},
    \end{align*}
    \end{minipage}
    }
    \end{equation*}
    where $[\yb_{t+1}^{\Kb}(\Mb|\Ab^{\prime},\Bb^{\prime},\{\wb^{\prime}\}), \vb_{t+1}^{\Kb}(\Mb)|\Ab^{\prime},\Bb^{\prime},\{\wb^{\prime}\}]_k$ denote the $k$-th element of the concatenated vector \\$[\yb_{t+1}^{\Kb}(\Mb|\Ab^{\prime},\Bb^{\prime},\{\wb^{\prime}\}), \vb_{t+1}^{\Kb}(\Mb)|\Ab^{\prime},\Bb^{\prime},\{\wb^{\prime}\}]$.\\\\ 
    To bound the term $\norm{\diff{[\yb_{t+1}^{\Kb}(\Mb|\Ab^{\prime},\Bb^{\prime},\{\wb^{\prime}\}), \vb_{t+1}^{\Kb}(\Mb|\Ab^{\prime},\Bb^{\prime},\{\wb^{\prime}\})]_k}{\Mb}}$ for all $k$, based on the expression of $\yb_{t+1}^{\Kb}(\Mb|\Ab^{\prime},\Bb^{\prime},\{\wb^{\prime}\})$ and $\vb_{t+1}^{\Kb}(\Mb|\Ab^{\prime},\Bb^{\prime},\{\wb^{\prime}\})$ we have\\

        {\bf For $\yb_{t+1}^{\Kb}(\Mb|\Ab^{\prime},\Bb^{\prime},\{\wb^{\prime}\})$:}
        \begin{equation*}
            \resizebox{\linewidth}{!}
            {
            \begin{minipage}{\linewidth}
            \begin{align*}
                \norm{\diff{[\yb_{t+1}^{\Kb}(\Mb|\Ab^{\prime},\Bb^{\prime},\{\wb^{\prime}\})]_k}{\Mb}} = \norm{\left[\frac{\partial[\yb_{t+1}^{\Kb}(\Mb|\Ab^{\prime},\Bb^{\prime},\{\wb^{\prime}\})]_k}{\partial\Mb^{[0]}}, \ldots, \frac{\partial[\yb_{t+1}^{\Kb}(\Mb|\Ab^{\prime},\Bb^{\prime},\{\wb^{\prime}\})]_k}{\partial\Mb^{[H-1]}}\right]} 
            \end{align*}
            \end{minipage}
            }
        \end{equation*}
        where $\norm{\frac{\partial[\yb_{t+1}^{\Kb}(\Mb|\Ab^{\prime},\Bb^{\prime},\{\wb^{\prime}\})]_k}{\partial\Mb^{[r]}}}$ for all $r$ is bounded as follows
        \begin{equation*}
        \begin{split}
            \norm{\frac{\partial[\yb_{t+1}^{\Kb}(\Mb|\Ab^{\prime},\Bb^{\prime},\{\wb^{\prime}\})]_k}{\partial\Mb^{[r]}}} &= \norm{\frac{\partial[\sum_{i=0}^{2H} \sum_{j=0}^H (\Ab^{\prime} - \Bb^{\prime}\Kb)^j\Bb^{\prime}\Mb^{[i-j-1]}\mathbb{1}_{i-j\in [1,H]} \wb^{\prime}_{t-i}]_k}{\partial\Mb^{[r]}}}\\
            &= \norm{\frac{\partial[\sum_{i=0}^{2H} \sum_{i-j=0}^H (\Ab^{\prime}-\Bb^{\prime}\Kb)^{i-j}\Bb^{\prime}\Mb^{[j-1]}\mathbb{1}_{j\in [1,H]} \wb^{\prime}_{t-i}]_k}{\partial\Mb^{[r]}}}\\
            &\leq \sum_{i=0}^{2H} \sum_{i-j=0}^H \left[\norm{(\Ab^{\prime}-\Bb^{\prime}\Kb)^{i-j}}\norm{\Bb^{\prime}}\norm{\wb^{\prime}_{t-i}}\mathbb{1}_{j-1=r}\mathbb{1}_{j\in [1,H]}\right]\\
            &\leq \kappa_{B^{\prime}} W^{\prime} \sum_{i=0}^H\norm{(\Ab^{\prime}-\Bb^{\prime}\Kb)^{i}}\leq \frac{(\kappa')^2\kappa_{B^{\prime}} W^{\prime}}{\gamma'}.
        \end{split}
        \end{equation*}

        {\bf For $\vb_{t+1}^{\Kb}(\Mb|\Ab^{\prime},\Bb^{\prime},\{\wb^{\prime}\})$:}
        \begin{equation*}
            \resizebox{\linewidth}{!}
            {
            \begin{minipage}{\linewidth}
            \begin{align*}
                \norm{\diff{[\vb_{t+1}^{\Kb}(\Mb|\Ab^{\prime},\Bb^{\prime},\{\wb^{\prime}\})]_k}{\Mb}} = \norm{\left[\frac{\partial[\vb_{t+1}^{\Kb}(\Mb|\Ab^{\prime},\Bb^{\prime},\{\wb^{\prime}\})]_k}{\partial\Mb^{[0]}}, \ldots, \frac{\partial[\vb_{t+1}^{\Kb}(\Mb|\Ab^{\prime},\Bb^{\prime},\{\wb^{\prime}\})]_k}{\partial\Mb^{[H-1]}}\right]} 
            \end{align*}
            \end{minipage}
            }
        \end{equation*}
        where $\norm{\frac{\partial[\vb_{t+1}^{\Kb}(\Mb|\Ab^{\prime},\Bb^{\prime},\{\wb^{\prime}\})]_k}{\partial\Mb^{[r]}}}$ for all $r$ is bounded as follows
        \begin{equation*}
        \begin{split}
            \norm{\frac{\partial[\vb_{t+1}^{\Kb}(\Mb|\Ab^{\prime},\Bb^{\prime},\{\wb^{\prime}\})]_k}{\partial\Mb^{[r]}}} &= \norm{\frac{\partial[-\Kb\yb_{t+1}^{\Kb}(\Mb|\Ab^{\prime},\Bb^{\prime},\{\wb^{\prime}\}) + \sum_{i=1}^H \Mb^{[i-1]} \wb^{\prime}_{t+1-i}]_k}{\partial\Mb^{[r]}}}\\
            &\leq \frac{(\kappa')^3\kappa_{B^{\prime}} W^{\prime}}{\gamma'} + W^{\prime}.
        \end{split}    
        \end{equation*}

        Based on the results, we have
        \begin{equation*}
        \begin{split}
            &\norm{\diff{[\yb_{t+1}^{\Kb}(\Mb|\Ab^{\prime},\Bb^{\prime},\{\wb^{\prime}\}), \vb_{t+1}^{\Kb}(\Mb|\Ab^{\prime},\Bb^{\prime},\{\wb^{\prime}\})]_k}{\Mb}}\\
            \leq &\sqrt{H} \max_r \norm{\frac{\partial[\yb_{t+1}^{\Kb}(\Mb|\Ab^{\prime},\Bb^{\prime},\{\wb^{\prime}\}), \vb_{t+1}^{\Kb}(\Mb|\Ab^{\prime},\Bb^{\prime},\{\wb^{\prime}\})]_k}{\partial\Mb^{[r]}}}\leq \sqrt{H} (\frac{(\kappa')^3\kappa_{B^{\prime}} W^{\prime}}{\gamma'} + W^{\prime}),
        \end{split}    
        \end{equation*}
        from which we get the final result
        \begin{equation*}
        \resizebox{\linewidth}{!}
        {
        \begin{minipage}{\linewidth}
        \begin{align*}
            &\norm{\nabla_{\Mb}f_{i,t+1}(\Mb|\Ab^{\prime},\Bb^{\prime},\{\wb^{\prime}\})}_F\\ 
            \leq &\text{rank}(\nabla_{\Mb}f_{i,t+1}(\Mb|\Ab^{\prime},\Bb^{\prime},\{\wb^{\prime}\})) \norm{\nabla_{\Mb}f_{i,t+1}(\Mb|\Ab^{\prime},\Bb^{\prime},\{\wb^{\prime}\})}\\ 
            \leq &d \sum_k\norm{\frac{\partial f_{i,t+1}(\Mb|\Ab^{\prime},\Bb^{\prime},\{\wb^{\prime}\})}{\partial[\yb_{t+1}^{\Kb}(\Mb|\Ab^{\prime},\Bb^{\prime},\{\wb^{\prime}\}), \vb_{t+1}^{\Kb}(\Mb|\Ab^{\prime},\Bb^{\prime},\{\wb^{\prime}\})]_k}}\norm{\diff{[\yb_{t+1}^{\Kb}(\Mb|\Ab^{\prime},\Bb^{\prime},\{\wb^{\prime}\}), \vb_{t+1}^{\Kb}(\Mb|\Ab^{\prime},\Bb^{\prime},\{\wb^{\prime}\})]_k}{\Mb}}\\
            \leq &d \sum_k\norm{\frac{\partial f_{t+1}(\Mb)}{\partial[\yb_{t+1}^{\Kb}(\Mb|\Ab^{\prime},\Bb^{\prime},\{\wb^{\prime}\}), \vb_{t+1}^{\Kb}(\Mb|\Ab^{\prime},\Bb^{\prime},\{\wb^{\prime}\})]_k}} \sqrt{H} (\frac{(\kappa')^3\kappa_{B^{\prime}} W^{\prime}}{\gamma'} + W^{\prime})\\
            \leq &G_c D \sqrt{d}d\sqrt{H} (\frac{(\kappa')^3\kappa_{B^{\prime}} W^{\prime}}{\gamma'} + W^{\prime}),
        \end{align*}
        \end{minipage}
        }
        \end{equation*}
        where the last inequality is due to the assumption on function gradients within the ball of radius of $D$.
  
    \end{proof}

    \begin{lemma}\label{L: Consecutive distance}
    Let Algorithm \ref{alg:D-online control} run with step size $\eta>0$ and define $\Mb_t = \frac{1}{m}\sum_{i=1}^m\Mb_{i,t}$. Under Assumptions \ref{A1}, \ref{A2} and \ref{A3}, we have that $\forall i \in [m]$    
    \begin{equation*}
        \norm{\Mb_{t} - \Mb_{i,t}}_F\leq \frac{2\eta G_c D \sqrt{d}d\sqrt{H} (\frac{\kappa^4 W}{\gamma} + W)\sqrt{m}}{1-\beta}
    \end{equation*}
    and
    \begin{equation*}
        \norm{\Mb_t - \widehat{\Mb}_{i,t+1}}_F\leq \frac{2\eta G_c D \sqrt{d}d\sqrt{H} (\frac{\kappa^4 W}{\gamma} + W)\sqrt{m}}{1-\beta},
    \end{equation*}
    where $D = \frac{W^{}(\kappa^2 + 2H\kappa^6)}{\gamma(1-\kappa^2(1-\gamma)^{(H+1)})} + \frac{2W\kappa^3}{\gamma}$.
    \end{lemma}
    \begin{proof}
    In this proof, we consider the vectorized versions of $\Mb_{i,t}$, $\Mb_{t}$ and $\nabla f_{i,t}(\Mb_{i,t}|\Ab,\Bb,\{\wb\})$. For simplicity, we define the following matrices
    \begin{equation*}
    \begin{split}
        \mathbb{M}_t &= [\Mb_{1,t},\ldots,\Mb_{m,t}]\\
        \widehat{\mathbb{M}}_t &= [\widehat{\Mb}_{1,t},\ldots,\widehat{\Mb}_{m,t}]\\
        \Gb_t &= [\nabla f_{1,t}(\Mb_{1,t}|\Ab,\Bb,\{\wb\}),\ldots,\nabla f_{m,t}(\Mb_{m,t}|\Ab,\Bb,\{\wb\})]\\
        \Rb_t &= [r_{1,t},\ldots,r_{m,t}],
    \end{split}
    \end{equation*}
    where $r_{i,t} = \widehat{\Mb}_{i,t} - \Mb_{i,t}$. Then the update can be expressed as $\mathbb{M}_t = \mathbb{M}_{t-1}\Pb - \eta \Gb_{t-1} - \Rb_t$.

    Expanding the update recursively, we get
    \begin{equation*}
        \mathbb{M}_t = \mathbb{M}_1\Pb^{t-1} - \sum_{l=1}^{t-1}\eta \Gb_{t-l}\Pb^{l-1} - \sum_{l=0}^{t-2} \Rb_{t-l}\Pb^{l}.
    \end{equation*}
    Since $\Pb$ is doubly stochastic, we have $\Pb^k\mathbf{1}=\mathbf{1}$ for all $k\geq 1$. Based on the geometric mixing bound of $\Pb$, we get
    \begin{equation*}
    \begin{split}
        &\norm{\Mb_{t} - \Mb_{i,t}}\\
        = &\norm{\mathbb{M}_t(\frac{1}{m}\mathbf{1}-\eb_i)}\\
        \leq &\norm{\mathbb{M}_{1}-\mathbb{M}_{1}[\Pb^{(t-1)}]_{:,i}} + \eta\sum_{l=1}^{t-1}\norm{\Gb_{t-l}(\frac{1}{m}\mathbf{1}-[\Pb^{l-1}]_{:,i})}
        +\sum_{l=0}^{t-2}\norm{\Rb_{t-l}(\frac{1}{m}\mathbf{1}-[\Pb^{l}]_{:,i})}\\
        \leq & \eta G_c D \sqrt{d}d\sqrt{H} (\frac{\kappa^4 W}{\gamma} + W)\sum_{l=1}^{t-1}\sqrt{m}\beta^{l-1} + \eta G_c D \sqrt{d}d\sqrt{H} (\frac{\kappa^4 W}{\gamma} + W) \sum_{l=0}^{t-2}\sqrt{m}\beta^l\\
        \leq &\frac{2\eta G_c D \sqrt{d}d\sqrt{H} (\frac{\kappa^4 W}{\gamma} + W)\sqrt{m}}{1-\beta}, 
    \end{split}    
    \end{equation*}
    where the second inequality is due to Lemma \ref{L: Gradient bound for reparameterized functions} and the fact that $\norm{\mathbb{M}_{1}-\mathbb{M}_{1}[\Pb^{(t-1)}]_{:,i}}=0$ by the identical initialization. Also based on Lemma \ref{L: Gradient bound for reparameterized functions} we have
    \begin{equation*}
    \begin{split}
        \norm{r_{i,t}} &= \norm{\widehat{\Mb}_{i,t} - \Mb_{i,t}}\\
        &\leq \norm{ \big(\sum_j [\Pb^{}]_{ji}\Mb_{j,t-1}-\eta\nabla f_{i,t-1}(\Mb_{i,t-1})\big) - \sum_j [\Pb^{}]_{ji}\Mb_{j,t-1}}\\
        &\leq \eta G_c D \sqrt{d}d\sqrt{H} (\frac{\kappa^4 W}{\gamma} + W),
    \end{split}
    \end{equation*}
    where the first inequality is due to projection on convex sets.\\\\
    By the same token, 
    \begin{equation*}
    \begin{split}
        &\norm{\Mb_t - \widehat{\Mb}_{i,t+1}}=\norm{\frac{1}{m}\mathbb{M}_t \mathbf{1}-(\mathbb{M}_t\Pb-\eta \Gb_t)\eb_i}\\
        =&\norm{\mathbb{M}_t(\frac{1}{m}\mathbf{1}-\Pb\eb_i)+\eta \Gb_t\eb_i}\\
        \leq &\norm{\Mb_{1} - \mathbb{M}_{1}[\Pb^{t}]_{:,i}} + \eta\sum_{l=1}^{t-1}\norm{\Gb_{t-l}(\frac{1}{m}\mathbf{1}-[\Pb^{l}]_{:,i})}\\
        +&\sum_{l=0}^{t-2}\norm{R_{t-l}(\frac{1}{m}\mathbf{1}-[\Pb^{l+1}]_{:,i})} + \norm{\eta \nabla f_{i,t}(\Mb_{i,t})}\\
        \leq &\eta G_c D \sqrt{d}d\sqrt{H} (\frac{\kappa^4 W}{\gamma} + W)\\
        + &\eta G_c D \sqrt{d}d\sqrt{H} (\frac{\kappa^4 W}{\gamma} + W)\left[\sum_{l=1}^{t-1}\sqrt{m}\beta^l + \sum_{l=0}^{t-2}\sqrt{m}\beta^{l+1}\right]\\
        \leq &\frac{2\eta G_c D \sqrt{d}d\sqrt{H} (\frac{\kappa^4 W}{\gamma} + W)\sqrt{m}}{1-\beta}.
    \end{split}    
    \end{equation*}

    \end{proof}

    \begin{theorem}\label{T: Regret bound for OCO with memory}(Regret bound for OCO with memory)
    Let Algorithm \ref{alg:D-online control} run with step size $\eta>0$. Under Assumptions \ref{A1}, \ref{A2} and \ref{A3}, we have that $\forall j \in [m]$   
    \begin{equation*}
    \begin{split}
        &\sum_t\sum_i f_{i,t}(\Mb_{j,t-H-1:t}) - f_{i,t}(\Mb^*)\\
        \leq &O(T\eta m\sqrt{m}\frac{\big(G_c D d^2H(\frac{\kappa^4 W}{\gamma} + W)\big)^2}{1-\beta}) + O\big(mT\eta H^3\sqrt{H}(G_c Dd)^2\kappa^4 W\frac{ (\frac{\kappa^4 W}{\gamma} + W)\sqrt{m}}{1-\beta}\big)\\
        + &O\bigg(\eta Tm\big(G_c D d^2H (\frac{\kappa^4 W}{\gamma} + W)\big)^2\bigg) + O(\frac{1}{\eta}) + O\big(\frac{1}{\eta}Tm(\frac{\eta G_c D d^2H (\frac{\kappa^4 W}{\gamma} + W)\sqrt{m}}{1-\beta})^2\big),
    \end{split}
    \end{equation*}
    where $D = \frac{W^{}(\kappa^2 + 2H\kappa^6)}{\gamma(1-\kappa^2(1-\gamma)^{(H+1)})} + \frac{2W\kappa^3}{\gamma}$.
    \end{theorem}
    \begin{proof}
    For simplicity, we denote $f_{i,t}(\cdot|\Ab,\Bb,\{\wb\})$ as $f_{i,t}(\cdot)$. Based on the update rule, $\forall i\in [m]$ and $t\in[1,\ldots, T]$ we have
    \begin{equation}\label{Eq1: Regret bound for OCO with memory}
    \resizebox{\linewidth}{!}
    {
    \begin{minipage}{\linewidth}
    $
     \begin{aligned}
        &f_{i,t}(\Mb_{i,t}) - f_{i,t}(\Mb^*)\\
        \leq &\nabla f_{i,t}(\Mb_{i,t})^{\top}(\Mb_{i,t} - \Mb^*)\\
        = &\frac{1}{\eta}\left[\frac{1}{2}\eta^2\norm{\nabla f_{i,t}(\Mb_{i,t})}^2 + \frac{1}{2}\norm{\Mb_{i,t}-\Mb^*}^2 - \frac{1}{2}\norm{\Mb_{i,t}-\Mb^* - \eta \nabla f_{i,t}(\Mb_{i,t})}^2\right]\\
        = &\frac{1}{2}\eta\norm{\nabla f_{i,t}(\Mb_{i,t})}^2 + \frac{1}{2\eta}\norm{\Mb_{i,t}-\Mb^*}^2 - \frac{1}{2\eta}\norm{\Mb_{i,t} - \sum_j (\Pb_{ji}\Mb_{j.t}) + \sum_j (\Pb_{ji}\Mb_{j.t}) - \eta \nabla f_{i,t}(\Mb_{i,t})-\Mb^*}^2\\
        = &\frac{1}{2}\eta\norm{\nabla f_{i,t}(\Mb_{i,t})}^2 + \frac{1}{2\eta}\norm{\Mb_{i,t}-\Mb^*}^2\\
        - &\frac{1}{2\eta}\left[\norm{\Mb_{i,t} - \sum_j (\Pb_{ji}\Mb_{j.t})}^2 + 2\big(\Mb_{i,t} - \sum_j (\Pb_{ji}\Mb_{j.t})\big)^{\top}(\widehat{\Mb}_{i,t+1} - \Mb^*) + \norm{\widehat{\Mb}_{i,t+1} - \Mb^*}^2\right]\\
        \leq &\frac{1}{2}\eta\norm{\nabla f_{i,t}(\Mb_{i,t})}^2 + \frac{1}{2\eta}\norm{\Mb_{i,t}-\Mb^*}^2 - \frac{1}{2\eta}\norm{\Mb_{i,t+1} - \Mb^*}^2\\
        - &\frac{1}{\eta}\big(\Mb_{i,t} - \sum_j (\Pb_{ji}\Mb_{j.t})\big)^{\top}(\widehat{\Mb}_{i,t+1} - \widehat{\Mb}_{t+1}) - \frac{1}{\eta}\big(\Mb_{i,t} - \sum_j (\Pb_{ji}\Mb_{j.t})\big)^{\top}(\widehat{\Mb}_{t+1} - \Mb^*),
    \end{aligned}   
    $
    \end{minipage}
    }
    \end{equation}
    where $\widehat{\Mb}_{t} = \frac{1}{m}\sum_i \widehat{\Mb}_{i,t}$ and the second inequality is due to the projection such that $\norm{\Mb_{i,t+1} - \Mb^*}\leq \norm{\widehat{\Mb}_{i,t+1} - \Mb^*}$.
    
    Based on Equation \eqref{Eq1: Regret bound for OCO with memory} and Lemmas \ref{L: Gradient bound for reparameterized functions} \& \ref{L: Consecutive distance}, we have
    \begin{equation}\label{Eq2: Regret bound for OCO with memory}
    \begin{split}
        &f_{i,t}(\Mb_{j,t}) - f_{i,t}(\Mb^*)\\
        =& f_{i,t}(\Mb_{j,t}) - f_{i,t}(\Mb_{i,t}) + f_{i,t}(\Mb_{i,t}) - f_{i,t}(\Mb^*)\\
        \leq& G_c D \sqrt{d}d\sqrt{H} (\frac{\kappa^4 W}{\gamma} + W)\norm{\Mb_{j,t} - \Mb_{i,t}} + f_{i,t}(\Mb_{i,t}) - f_{i,t}(\Mb^*)\\
        \leq & G_c D \sqrt{d}d\sqrt{H} (\frac{\kappa^4 W}{\gamma} + W)\frac{2\eta G_c D \sqrt{d}d\sqrt{H} (\frac{\kappa^4 W}{\gamma} + W)\sqrt{m}}{1-\beta}\\
        + &\frac{1}{2}\eta\norm{\nabla f_{i,t}(\Mb_{i,t})}^2 + \frac{1}{2\eta}\norm{\Mb_{i,t}-\Mb^*}^2 - \frac{1}{2\eta}\norm{\Mb_{i,t+1} - \Mb^*}^2\\
    - &\frac{1}{\eta}\big(\Mb_{i,t} - \sum_j (\Pb_{ji}\Mb_{j.t})\big)^{\top}(\widehat{\Mb}_{i,t+1} - \widehat{\Mb}_{t+1}) - \frac{1}{\eta}\big(\Mb_{i,t} - \sum_j (\Pb_{ji}\Mb_{j.t})\big)^{\top}(\widehat{\Mb}_{t+1} - \Mb^*).
    \end{split}    
    \end{equation}
    Summing Equation \eqref{Eq2: Regret bound for OCO with memory} over $i$, we get
    \begin{equation}\label{Eq3: Regret bound for OCO with memory}
    \begin{split}
        &\sum_i f_{i,t}(\Mb_{j,t}) - f_{i,t}(\Mb^*)\\
        \leq & mG_c D \sqrt{d}d\sqrt{H} (\frac{\kappa^4 W}{\gamma} + W)\frac{2\eta G_c D \sqrt{d}d\sqrt{H} (\frac{\kappa^4 W}{\gamma} + W)\sqrt{m}}{1-\beta}\\
        + &\frac{1}{2}\eta\sum_i\norm{\nabla f_{i,t}(\Mb_{i,t})}^2 
        + \frac{1}{2\eta}\sum_i\left[\norm{\Mb_{i,t}}^2-\norm{\Mb_{i,t+1}}^2 + 2(\Mb_{i,t+1} - \Mb_{i,t})^{\top}\Mb^*\right]\\    
    - &\frac{1}{\eta}\sum_i\big(\Mb_{i,t} - \sum_j (\Pb_{ji}\Mb_{j.t})\big)^{\top}(\widehat{\Mb}_{i,t+1} - \widehat{\Mb}_{t+1}).
    \end{split}    
    \end{equation}
    Summing Equation \eqref{Eq3: Regret bound for OCO with memory} over $t\in[1,\ldots,T]$, we have
    \begin{equation}\label{Eq4: Regret bound for OCO with memory}
    \begin{split}
        &\sum_t\sum_i f_{i,t}(\Mb_{j,t}) - f_{i,t}(\Mb^*)\\
        \leq & 2T\eta m\sqrt{m}\frac{\big(G_c D \sqrt{d}d\sqrt{H} (\frac{\kappa^4 W}{\gamma} + W)\big)^2}{1-\beta}        
        + \frac{1}{2}\eta\sum_t\sum_i\norm{\nabla f_{i,t}(\Mb_{i,t})}^2 
        + \frac{1}{2\eta}\sum_i\norm{\Mb_{i,1}}^2\\      
        + &\frac{1}{\eta}\sum_i\left(\Mb_{i,T+1}^{\top} - \Mb_{i,1}^{\top}\right)\Mb^*
        + \frac{1}{\eta}\sum_t\sum_i \norm{\big(\Mb_{i,t} - \sum_j (\Pb_{ji}\Mb_{j.t})\big)}\norm{(\widehat{\Mb}_{i,t+1} - \widehat{\Mb}_{t+1})}.       
    \end{split}
    \end{equation}
    Next, based on Lemma \ref{L: Lipschitz constant in terms of the framework of OCO with memory} and Lemma \ref{L: Consecutive distance}, we have
    \begin{equation}\label{Eq5: Regret bound for OCO with memory}
    \begin{split}
        &f_{i,t}(\Mb_{j,t-H-1:t}) - f_{i,t}(\Mb_{j,t}) = f_{i,t}(\Mb_{j,t-H-1},\ldots,\Mb_{j,t}) - f_{i,t}(\Mb_{j,t},\ldots,\Mb_{j,t})\\
        \leq &3G_c D \kappa^4W \sum_{k=1}^{H+1}\sum_{l=0}^{H-1}\norm{\Mb_{j,t-k}^{[l]} - \Mb_{j,t}^{[l]}}\\
        \leq &3G_c D \kappa^4W \sum_{l=0}^{H-1}\sum_{k=1}^{H+1}\sum_{s=0}^{k-1}\norm{\Mb_{j,t-s-1}^{[l]} - \Mb_{j,t-s}^{[l]}}_F\\
        \leq &3G_c D \kappa^4W \sum_{k=1}^{H+1}\sum_{s=0}^{k-1} 
        \sqrt{H} \norm{\Mb_{j,t-s-1} - \Mb_{j,t-s}}_F\\
        \leq &3G_c D \kappa^4W  \sum_{k=1}^{H+1}\sum_{s=0}^{k-1} 
        \sqrt{H} \norm{\Mb_{j,t-s-1} - \Mb_{t-s-1} + \Mb_{t-s-1} - \Mb_{j,t-s}}_F\\
        \leq &3G_c D \kappa^4W  \sum_{k=1}^{H+1}\sum_{s=0}^{k-1} 
        \left[\sqrt{H} \frac{4\eta G_c D \sqrt{d}d\sqrt{H} (\frac{\kappa^4 W}{\gamma} + W)\sqrt{m}}{1-\beta}\right]\\
        \leq &O\big(\eta H^3(G_c D)^2\sqrt{d}d\kappa^4 W\frac{ (\frac{\kappa^4 W}{\gamma} + W)\sqrt{m}}{1-\beta}\big).
    \end{split}    
    \end{equation}
    Summing Equation \eqref{Eq5: Regret bound for OCO with memory} over $i,t$ and combining it with Equation \eqref{Eq4: Regret bound for OCO with memory}, we have
    \begin{equation}\label{Eq6: Regret bound for OCO with memory}
    \begin{aligned}
    &\sum_t\sum_i f_{i,t}(\Mb_{j,t-H-1:t}) - f_{i,t}(\Mb^*)\\
    \leq & 2T\eta m\sqrt{m}\frac{\big(G_c D \sqrt{d}d\sqrt{H} (\frac{\kappa^4 W}{\gamma} + W)\big)^2}{1-\beta} + O\big(mT\eta H^3(G_c D)^2\sqrt{d}d\kappa^4 W\frac{ (\frac{\kappa^4 W}{\gamma} + W)\sqrt{m}}{1-\beta}\big)\\        
    + &\frac{1}{2}\eta\sum_t\sum_i\norm{\nabla f_{i,t}(\Mb_{i,t})}^2 
    + \frac{1}{2\eta}\sum_i\norm{\Mb_{i,1}}^2      
    + \frac{1}{\eta}\sum_i\left(\Mb_{i,T+1}^{\top} - \Mb_{i,1}^{\top}\right)\Mb^*\\
    + &\frac{1}{\eta}\sum_t\sum_i \norm{\big(\Mb_{i,t} - \sum_j (\Pb_{ji}\Mb_{j.t})\big)}\norm{(\widehat{\Mb}_{i,t+1} - \widehat{\Mb}_{t+1})}\\
    \leq &O(T\eta m\sqrt{m}\frac{\big(G_c D \sqrt{d}d\sqrt{H} (\frac{\kappa^4 W}{\gamma} + W)\big)^2}{1-\beta}) + O\big(mT\eta H^3(G_c D)^2\sqrt{d}d\kappa^4 W\frac{ (\frac{\kappa^4 W}{\gamma} + W)\sqrt{m}}{1-\beta}\big)\\
    + &O\bigg(\eta Tm\big(G_c D \sqrt{d}d\sqrt{H} (\frac{\kappa^4 W}{\gamma} + W)\big)^2\bigg) + O(\frac{1}{\eta}) + O\big(\frac{1}{\eta}Tm(\frac{\eta G_c D \sqrt{d}d\sqrt{H} (\frac{\kappa^4 W}{\gamma} + W)\sqrt{m}}{1-\beta})^2\big),
    \end{aligned}
    \end{equation}
    where the last inequality is based on Lemmas \ref{L: Gradient bound for reparameterized functions} and \ref{L: Consecutive distance}.
    \end{proof}

\begin{proof}[Proof of Theorem \ref{T: Regret bound (known case)}]
Suppose $\Kb^*$ is the optimal linear policy in hindsight, and we decompose the individual regret into the following terms
\begin{equation}\label{Eq1: Regret bound (known case)}
\begin{split}
    &J_T^j(\mathcal{A}_1) - \min_{\Kb\in\mathcal{K}}J_T(\Kb)\\
    =&\sum_{t=1}^{T}\sum_{i=1}^m\left[c_{i,t}\big(\xb^{\Kb}_{j,t}(\Mb_{j,1:t-1}), \ub^{\Kb}_{j,t}(\Mb_{j,1:t})\big) - c_{i,t}\big(\xb^{\Kb^*}_t(0),\ub^{\Kb^*}_t(0)\big)\right]\\
    =&\sum_{t=1}^{T}\sum_{i=1}^m\left[c_{i,t}\big(\xb^{\Kb}_{j,t}(\Mb_{j,1:t-1}), \ub^{\Kb}_{j,t}(\Mb_{j,1:t})\big) - c_{i,t}\big(\yb^{\Kb}_{j,t}(\Mb_{j,t-H-1:t}), \vb^{\Kb}_{j,t}(\Mb_{j,t-H-1:t})\big)\right]\\
    +&\sum_{t=1}^{T}\sum_{i=1}^m\left[c_{i,t}\big(\yb^{\Kb}_{j,t}(\Mb_{j,t-H-1:t}), \vb^{\Kb}_{j,t}(\Mb_{j,t-H-1:t}) - c_{i,t}\big(\yb^{\Kb}_{t}(\Mb^*), \vb^{\Kb}_{t}(\Mb^*)\big)\right]\\
    +&\sum_{t=1}^{T}\sum_{i=1}^m\left[c_{i,t}\big(\yb^{\Kb}_{t}(\Mb^*), \vb^{\Kb}_{t}(\Mb^*)\big) - c_{i,t}\big(\xb^{\Kb^*}_t(\Mb^*),\ub^{\Kb^*}_t(\Mb^*)\right]\\
    +&\sum_{t=1}^{T}\sum_{i=1}^m\left[c_{i,t}\big(\xb^{\Kb^*}_t(\Mb^*),\ub^{\Kb^*}_t(\Mb^*)-c_{i,t}\big(\xb^{\Kb^*}_t(0),\ub^{\Kb^*}_t(0)\big)\right],
\end{split}    
\end{equation}
where each term is bounded as follows:

1. Based on Lemma \ref{L: Upper bounds of (surrogate)states and (surrogate)actions and the difference}, we have that 
\begin{equation}\label{Eq2: Regret bound (known case)}
    \norm{\xb^{\Kb}_{j,t}(\Mb_{j,1:t-1}) - \yb^{\Kb}_{j,t}(\Mb_{j,t-H-1:t})},\norm{\ub^{\Kb}_{j,t}(\Mb_{j,1:t}) - \vb^{\Kb}_{j,t}(\Mb_{j,t-H-1:t})}\leq \kappa^3(1-\gamma)^{H+1}D,
\end{equation}
where $D$ is defined as $\frac{W^{}(\kappa^2 + 2H\kappa^6)}{\gamma(1-\kappa^2(1-\gamma)^{(H+1)})} + \frac{2W\kappa^3}{\gamma}$. Then based on Equation \eqref{Eq2: Regret bound (known case)} and Assumption \ref{A1}, we have
\begin{equation}\label{Eq3: Regret bound (known case)}
\begin{split}
    &c_{i,t}\big(\xb^{\Kb}_{j,t}(\Mb_{j,1:t-1}), \ub^{\Kb}_{j,t}(\Mb_{j,1:t})\big) - c_{i,t}\big(\yb^{\Kb}_{j,t}(\Mb_{j,t-H-1:t}), \vb^{\Kb}_{j,t}(\Mb_{j,t-H-1:t})\big)\\
    \leq &G_c D\sqrt{\norm{\xb^{\Kb}_{j,t}(\Mb_{j,1:t-1}) - \yb^{\Kb}_{j,t}(\Mb_{j,t-H-1:t})}^2 + \norm{\ub^{\Kb}_{j,t}(\Mb_{j,1:t}) - \vb^{\Kb}_{j,t}(\Mb_{j,t-H-1:t})}^2}\\
    \leq & G_c D\sqrt{2}\kappa^3(1-\gamma)^{H+1}D.
\end{split}
\end{equation}
Summing Equation \eqref{Eq3: Regret bound (known case)} over $i$ and $t$, we get
\begin{equation}\label{Eq4: Regret bound (known case)}
\begin{split}
    &\sum_{t=1}^{T}\sum_{i=1}^m\left[c_{i,t}\big(\xb^{\Kb}_{j,t}(\Mb_{j,1:t-1}), \ub^{\Kb}_{j,t}(\Mb_{j,1:t})\big) - c_{i,t}\big(\yb^{\Kb}_{j,t}(\Mb_{j,t-H-1:t}), \vb^{\Kb}_{j,t}(\Mb_{j,t-H-1:t})\big)\right] \\
    \leq &Tm G_c D\sqrt{2}\kappa^3(1-\gamma)^{H+1}D.
\end{split}
\end{equation}
Based on the same process, we also have
\begin{equation}\label{Eq5: Regret bound (known case)}
    \sum_{t=1}^{T}\sum_{i=1}^m\left[c_{i,t}\big(\yb^{\Kb}_{t}(\Mb^*), \vb^{\Kb}_{t}(\Mb^*)\big) - c_{i,t}\big(\xb^{\Kb^*}_t(\Mb^*),\ub^{\Kb^*}_t(\Mb^*)\right] \leq Tm G_c D\sqrt{2}\kappa^3(1-\gamma)^{H+1}D.
\end{equation}
\\
2. Based on Lemma \ref{L: Approximate a linear policy using time-invariant DFC policy}, we have
\begin{equation}\label{Eq6: Regret bound (known case)}
    \sum_{t=1}^{T}\sum_{i=1}^m\left[c_{i,t}\big(\xb^{\Kb^*}_t(\Mb^*),\ub^{\Kb^*}_t(\Mb^*)-c_{i,t}\big(\xb^{\Kb^*}_t(0),\ub^{\Kb^*}_t(0)\big)\right]\leq TmG_c D\frac{6WH \kappa^8(1-\gamma)^{H-1}}{\gamma}.
\end{equation}
\\
3. Based on Theorem \ref{T: Regret bound for OCO with memory}, we get
\begin{equation}\label{Eq7: Regret bound (known case)}
\resizebox{\linewidth}{!}
{
\begin{minipage}{\linewidth}
    $
    \begin{aligned}
    &\sum_{t=1}^{T}\sum_{i=1}^m\left[c_{i,t}\big(\yb^{\Kb}_{j,t}(\Mb_{j,t-H-1:t}), \vb^{\Kb}_{j,t}(\Mb_{j,t-H-1:t}) - c_{i,t}\big(\yb^{\Kb}_{t}(\Mb^*), \vb^{\Kb}_{t}(\Mb^*)\big)\right]\\
    \leq &O(T\eta m\sqrt{m}\frac{\big(G_c D \sqrt{d}d\sqrt{H} (\frac{\kappa^4 W}{\gamma} + W)\big)^2}{1-\beta}) + O\big(mT\eta H^3(G_c D)^2\sqrt{d}d\kappa^4 W\frac{ (\frac{\kappa^4 W}{\gamma} + W)\sqrt{m}}{1-\beta}\big)\\
    + &O\bigg(\eta Tm\big(G_c D \sqrt{d}d\sqrt{H} (\frac{\kappa^4 W}{\gamma} + W)\big)^2\bigg) + O(\frac{1}{\eta}) + O\big(\frac{1}{\eta}Tm(\frac{\eta G_c D \sqrt{d}d\sqrt{H} (\frac{\kappa^4 W}{\gamma} + W)\sqrt{m}}{1-\beta})^2\big).
    \end{aligned}
    $
\end{minipage}

}
\end{equation}

Combining Equations \eqref{Eq4: Regret bound (known case)}, \eqref{Eq5: Regret bound (known case)}, \eqref{Eq6: Regret bound (known case)} and \eqref{Eq7: Regret bound (known case)}, the order of the result is proved by the selections of $\eta$ and $H$.
\end{proof}

\subsection{System Identification:}
For the unknown case, since the learning process is built on agents' system estimates, it is necessary to quantify the estimation error in order to see how it affects the final regret bound. Here we provide a series of results for system estimates, including quantifying how a strongly stable w.r.t. one system performs on another similar system (Lemma \ref{L:strong stability for a close system}), the system estimation error of each agent (Theorem \ref{T: System identification}) and the noise estimation error (Lemma \ref{L: Diff between estimated and real noise}).

\begin{lemma}\label{L:strong stability for a close system}
A linear policy $\Kb$ which is $(\kappa, \gamma)$-strongly stable ($\kappa\geq 1$) for the LTI system $(\Ab_1, \Bb_1)$, is also $(\kappa, \gamma-2\kappa^3\epsilon)$-strongly stable for the LTI system $(\Ab_2, \Bb_2)$ if $\norm{\Ab_1 - \Ab_2\}}, \norm{\Bb_1 - \Bb_2}\leq \epsilon$ and $\epsilon < \frac{\gamma}{2\kappa^3}$.\end{lemma}
\begin{proof}
    Based on the definition of strong stability, we have $\Ab_1 + \Bb_1\Kb=\Hb\Lb\Hb^{-1}$ with $\norm{\Hb},\norm{\Hb^{-1}}\leq \kappa$ and $\norm{\Lb}\leq 1-\gamma$. Therefore, we have
    \begin{equation*}
    \begin{split}
        &\Ab_2+\Bb_2\Kb\\
        =&\Ab_1+\Bb_1\Kb + (\Ab_2 - \Ab_1) + (\Bb_2 - \Bb_1)\Kb\\
        =&\Hb\Lb\Hb^{-1} + \Hb\Hb^{-1}[(\Ab_2 - \Ab_1) + (\Bb_2 - \Bb_1)\Kb]\Hb\Hb^{-1}\\
        =&\Hb\big[\Lb + \Hb^{-1}[(\Ab_2 - \Ab_1) + (\Bb_2 - \Bb_1)\Kb]\Hb\big]\Hb^{-1}.
    \end{split}    
    \end{equation*}
    For the middle term
    , it can be shown that
    \begin{equation*}
    \begin{split}
        &\norm{\Lb + \Hb^{-1}[(\Ab_2 - \Ab_1) + (\Bb_2 - \Bb_1)\Kb]\Hb}\\
        \leq &\norm{\Lb} + \norm{\Hb^{-1}}\norm{(\Ab_2 - \Ab_1) + (\Bb_2 - \Bb_1)\Kb}\norm{\Hb}\\
        \leq & (1-\gamma) + \kappa^2(\norm{\Ab_2 - \Ab_1} + \norm{(\Bb_2 - \Bb_1)\Kb})\\
        \leq & (1-\gamma) + \kappa^2(\epsilon + \epsilon\kappa)\\
        \leq & (1-\gamma) + 2\kappa^3\epsilon.
    \end{split}    
    \end{equation*}
    Suppose $\gamma - 2\kappa^3\epsilon > 0$, based on the result above, we can see that the policy $\Kb$ is $(\kappa, \gamma - 2\kappa^3\epsilon)$-strongly stable w.r.t. the LTI system $(\Ab_2, \Bb_2)$ as $\Ab_2 + \Bb_2\Kb = \Hb\Lb_2\Hb^{-1}$ where $\Lb_2:=\Lb + \Hb^{-1}[(\Ab_2 - \Ab_1) + (\Bb_2 - \Bb_1)\Kb]\Hb$ and $\norm{\Lb_2}\leq 1-(\gamma - 2\kappa^3\epsilon)$. 
\end{proof}
\begin{lemma}\label{L: Series of moment differences}(Lemma 17 in \cite{hazan2020nonstochastic})
For a matrix pair $(\Mb,\Delta \Mb)$ such that $\norm{\Mb},\norm{\Mb+\Delta\Mb}\leq 1-\gamma$, where $\gamma <1$, we have
\begin{equation*}
    \norm{(\Mb+\Delta\Mb)^t - \Mb^t}\leq t(1-\gamma)^{t-1}\norm{\Delta\Mb}
\end{equation*}
and 
\begin{equation*}
    \sum_{t=0}^{\infty} \norm{(\Mb+\Delta\Mb)^t - \Mb^t}\leq \gamma^{-2}\norm{\Delta\Mb}.
\end{equation*}
\end{lemma}

\begin{lemma}\label{L: Bounds on states and actions during the exploration phase}
During the exploration phase where the control is chosen as $\ub_{i,t} = -\Kb\xb_{i,t} + \xi_{i,t}$, where $\xi_{i,t}\sim\{\pm 1\}^{d_2}$, it can be shown that for $t\leq T_0+1$, we have 
\begin{equation*}
    \norm{\xb_{i,t}}\leq \frac{2\kappa^3 W\sqrt{d_2}}{\gamma},\;\forall i\in[m].
\end{equation*}
\end{lemma}
\begin{proof}
    Based on Equation \eqref{Eq1: Diff between estimated and real noise}, we have for any $t\leq T_0+1$
    \begin{equation}\label{Eq1: Bounds on states and actions during the exploration phase}
    \begin{split}
        \norm{\xb_{t}} &= \norm{\sum_{i=1}^{t-1} (\Ab-\Bb\Kb)^{t-i}(\wb_i + \Bb\Tilde{\ub}_i)}\\
        &\leq \norm{\sum_{i=1}^{t-1} (\Ab-\Bb\Kb)^{t-i} \wb_i} + \sum_{i=1}^{t-1} \norm{(\Ab-\Bb\Kb)^{t-i}}\norm{\Bb}\sqrt{d_2}\\
        &\leq \frac{\kappa^2 W}{\gamma} + \frac{\kappa^3 \sqrt{d_2}}{\gamma}\\
        &\leq \frac{2\kappa^3 W\sqrt{d_2}}{\gamma}.
    \end{split}    
    \end{equation}
        
\end{proof}

\begin{lemma}\label{L: Martingale estimation} Let $\Nb_{k-1} = \frac{1}{m}\sum_i\Nb^{T_{collect}}_{i,k-1},\;\forall k\in [q+1]$. Suppose Assumptions \ref{A1} and \ref{A3} hold. By running Algorithm \ref{alg:D-online control (unknown system)}, we have with probability at least $(1-\delta)$,
\begin{equation*}
    \norm{\Nb^{T_0}_{i,k-1} - (\Ab-\Bb\Kb)^{k-1}\Bb}\leq \frac{2d_2\kappa^3W}{\gamma}\sqrt{\frac{8\log\big(d_1 d_2 (q+1)m\delta^{-1}\big)}{m\big(T_{collect}-(q+1)\big)}},\; \forall k\in[q+1] \text{ and } i\in[m].
\end{equation*} 
\end{lemma}
\begin{proof}
    Let $\Nb^{t}_{i,k-1} = \xb_{i,t+k}\xi^{\top}_{i,t} - (\Ab-\Bb\Kb)^{k-1}\Bb$. Based on Lemma 21 in \cite{hazan2020nonstochastic}, we know $\{\Nb^{t}_{i,k-1}\}_t$ is a martingale difference sequence w.r.t. $\{\xi_{i,t}\}_t$. Notice that each agent possesses an identical and independent martingale difference sequence. If we consider the concatenation of those sequences: $\{\{\Nb^{t}_{1,k-1}\}_t, \{\Nb^{t}_{2,k-1}\}_t, \ldots, \{\Nb^{t}_{m,k-1}\}_t\}$, we can see that sequence is also a martingale sequence w.r.t. $\{\{\xi_{1,t}\}_t, \{\xi_{2,t}\}_t,\ldots,\{\xi_{m,t}\}_t \}$. Then with the result of Lemma 21 in \cite{hazan2020nonstochastic}, we have with probability at least $(1-\delta)$
    \begin{equation}\label{Eq1: Martingale estimation}
        \norm{\Nb_{k-1} - (\Ab-\Bb\Kb)^{k-1}\Bb} \leq \frac{d_2\kappa^3W}{\gamma}\sqrt{\frac{8\log\big(d_1 d_2 (q+1)m\delta^{-1}\big)}{m\big(T_{collect}-(q+1)\big)}},\;\forall k\in[q+1],
    \end{equation}
    where $\Nb_{k-1} = \frac{1}{m}\sum_i\Nb^{T_{collect}}_{i,k-1}$ and $\Nb_{k-1} - (\Ab-\Bb\Kb)^{k-1}\Bb$ is the average over the elements of\\ $\{\{\Nb^{t}_{1,k-1}\}_t, \{\Nb^{t}_{2,k-1}\}_t, \ldots, \{\Nb^{t}_{m,k-1}\}_t\}$.\\\\

    Now based on the geometric bound of the network topology, we have
    \begin{equation}\label{Eq2: Martingale estimation}
    \begin{split}
        \norm{\Nb^{t}_{i,k-1} - \Nb_{k-1}} &= \norm{\sum_{j=1}^m [\Pb^{t-T_{collect}}]_{ji}\Nb^{T_{collect}}_{j,k-1} -  \frac{1}{m}\sum_j\Nb^{T_{collect}}_{j,k-1}}\\
        &\leq \sum_{j=1}^m |[\Pb^{t-T_{collect}}]_{ji} - \frac{1}{m}|\norm{\Nb^{T_{collect}}_{j,k-1}}\\
        &\leq \sqrt{m}\beta^{t-T_{collect}}\frac{2\kappa^3Wd_2}{\gamma},
    \end{split}
    \end{equation}
    where the last inequality is due to the fact that based on Lemma \ref{L: Bounds on states and actions during the exploration phase} we have $\norm{\xb_{i,t}}\leq \frac{2\kappa^3W\sqrt{d_2}}{\gamma}$ , and also $\norm{\xi_{i,t}}\leq \sqrt{d_2}$.

    Based on Equations \eqref{Eq1: Martingale estimation} and \eqref{Eq2: Martingale estimation}, we have $\forall i$
    \begin{equation}\label{Eq3: Martingale estimation}
    \begin{split}
        \norm{\Nb^{T_0}_{i,k-1} - (\Ab-\Bb\Kb)^{k-1}\Bb}&\leq \sqrt{m}\beta^{T_{exchange}}\frac{2\kappa^3Wd_2}{\gamma} + \frac{d_2\kappa^3W}{\gamma}\sqrt{\frac{8\log\big(d_1 d_2 (q+1)m\delta^{-1}\big)}{m\big(T_{collect}-(q+1)\big)}}\\
        &\leq \frac{2d_2\kappa^3W}{\gamma}\sqrt{\frac{8\log\big(d_1 d_2 (q+1)m\delta^{-1}\big)}{m\big(T_{collect}-(q+1)\big)}},        
    \end{split}
    \end{equation}
    where the second term in the first inequality can be the dominant one by specifically choosing the lengths of $T_{collect}$ and $T_{exchange}$. 
\end{proof}

\begin{lemma}\label{L: Diff between solutions of the actual and perturbed linear systems}(Lemma 22 in \cite{hazan2020nonstochastic})
Let $\xb^*$ be the solution to the linear system: $\Ab \xb = \bb$, and $\hat{\xb}$ be the solution to $(\Ab + \Delta \Ab)\xb = (\bb + \Delta\bb)$, then as long as $\norm{\Delta\Ab}<\sigma_{\text{min}}(\Ab)$, we have
\begin{equation*}
    \norm{\xb^* - \hat{\xb}} \leq \frac{\norm{\Delta \bb} +\norm{\Delta\Ab}\norm{\xb^*}}{\sigma_{\text{min}}(\Ab) - \norm{\Delta \Ab}}.
\end{equation*}
\end{lemma}

\begin{theorem}\label{T: System identification}
Suppose Assumptions \ref{A1}, \ref{A3} and \ref{A4} hold. By running Algorithm \ref{alg:D-online control (unknown system)}, we have with probability at least $(1-\delta)$,
\begin{equation*}
    \norm{\widehat{\Ab}_i - \Ab}, \norm{\widehat{\Bb}_i - \Bb} \leq \epsilon_{\Ab,\Bb},\;\forall i\in[m],
\end{equation*}
where $\epsilon_{\Ab,\Bb} = O\big(\frac{d_2W\sqrt{d_1 q}\kappa^{11/2}}{\gamma}\sqrt{\frac{\log\big(d_1 d_2 (q+1)m\delta^{-1}\big)}{m\big(T_{collect}-(q+1)\big)}}\big)$.
\end{theorem}
\begin{proof}
    Denote $(\Ab - \Bb\Kb)$ as $\Ab^{\prime}$. Based on Assumption , we know $\Ab^{\prime}$ is the unique solution $\Xb^*$ of the following linear system
    \begin{equation*}
        \Xb\Cb_q = [\Ab^{\prime}\Bb,(\Ab^{\prime})^2\Bb,\ldots,(\Ab^{\prime})^q\Bb] = \Ab^{\prime}\Cb_q.
    \end{equation*}
    Suppose $\forall i\in[m]\text{ and }k\in[q+1]$, $\norm{\Nb^{T_0}_{i,k-1} - (\Ab^{\prime})^{k-1}\Bb}\leq \epsilon$. Then it is implied that $\forall i$, $\norm{\Cb_{i,0} - \Cb_q},\norm{\Cb_{i,1} - \Ab^{\prime}\Cb_q}\leq \sqrt{q}\epsilon$ and $\norm{\widehat{\Bb}_i - \Bb}\leq \epsilon$. Note that $\widehat{\Ab}^{\prime}_i$ is the solution to the following perturbed system
    \begin{equation*}
        \Xb\Cb_{i,0} = \Cb_{i,1}.
    \end{equation*}
    By applying Lemma \ref{L: Diff between solutions of the actual and perturbed linear systems} on the rows of the system above, we have
    \begin{equation}\label{Eq1: System identification}
    \begin{split}
        \norm{\widehat{\Ab}^{\prime}_i - {\Ab}^{\prime}} &\leq \frac{\sqrt{d_1}(\norm{\Cb_{i,1} - \Ab^{\prime}\Cb_q} + \norm{\Cb_{i,0} - \Cb_q}\norm{\Ab^{\prime}})}{\sigma_{\text{min}}(\Cb_q) - \norm{\Cb_{i,0} - \Cb_q}}\\
        &\leq \frac{\sqrt{d_1}(\sqrt{q}\epsilon + \sqrt{q}\epsilon2\kappa^2)}{1/\sqrt{\kappa} - \sqrt{q}\epsilon}\\
        &\leq 6\sqrt{d_1 q}\kappa^{5/2}\epsilon,
    \end{split}
    \end{equation}
    where the last inequality is by assuming $\sqrt{q}\epsilon \leq \frac{1}{2\sqrt{\kappa}}$ (This can be achieved by specifically choosing $T_{collect}$ and $T_{exchange}$).
    Based on Equation \eqref{Eq1: System identification}, we have
    \begin{equation}\label{Eq2: System identification}
    \begin{split}
        \norm{\widehat{\Ab}_i - \Ab} &= \norm{(\widehat{\Ab}^{\prime}_i + \widehat{\Bb}_i\Kb) - (\Ab^{\prime} + \Bb\Kb)}\\
        &\leq  \norm{\widehat{\Ab}^{\prime}_i - {\Ab}^{\prime}} + \norm{\widehat{\Bb}_i\Kb - \Bb\Kb}\\
        &\leq 6\sqrt{d_1 q}\kappa^{5/2}\epsilon + \epsilon\kappa\leq 7\sqrt{d_1 q}\kappa^{5/2}\epsilon:=\epsilon_{\Ab,\Bb}.
    \end{split}    
    \end{equation}

    From Lemma \ref{L: Martingale estimation}, we have $\epsilon = O\big(\frac{d_2\kappa^3W}{\gamma}\sqrt{\frac{\log\big(d_1 d_2 (q+1)m\delta^{-1}\big)}{m\big(T_{collect}-(q+1)\big)}}\big)$, which concludes the claim.
\end{proof}

\begin{lemma}\label{L: Diff between estimated and real noise}Suppose Assumptions \ref{A1}, \ref{A3} and \ref{A4} hold. By running Algorithm \ref{alg:D-online control (unknown system)} with specifically chosen $T_{collect}$ and $T_{exchange}$ such that $\norm{\Ab - \widehat{\Ab}_i},\norm{\Bb - \widehat{\Bb}_i}\leq \epsilon\leq \frac{1}{12}\kappa^{-7}\gamma^2,\;\forall i\in[m]$, we have for any $t\geq T_0+1$, $\norm{\wb_t - \hat{\wb}_{i,t}}\leq \zeta\epsilon$, where $\zeta = \frac{24 W\sqrt{d_2}\kappa^7}{\gamma^2}$.
    
\end{lemma}
\begin{proof}
    Suppose the control is chosen as the form: $\ub_t = -\Kb\xb_t + \Tilde{\ub}_t$, then based on the linear dynamics we have
    \begin{equation}\label{Eq1: Diff between estimated and real noise}
        \xb_{t+1} = (\Ab-\Bb\Kb)^{t-t^{\prime}+1}\xb_{t^{\prime}} + \sum_{i=t^{\prime}}^{t} (\Ab-\Bb\Kb)^{t-i}(\wb_i + \Bb\Tilde{\ub}_i).
    \end{equation}
    In the learning phase, the control $\Tilde{\ub}_{i,t}$ of agent $i$ is chosen as
    \begin{equation}\label{Eq2: Diff between estimated and real noise}
        \Tilde{\ub}_{i,t} = \sum_{j=1}^H \Mb_{i,t}^{[j-1]} \hat{\wb}_{i,t-j}.
    \end{equation}
    Now based on Equations \eqref{Eq1: Diff between estimated and real noise} and \eqref{Eq2: Diff between estimated and real noise}, it can be shown that for $t\geq T_0+1$, $\norm{\xb_{i,t}}\leq $ and $\norm{\wb_t - \hat{\wb}_{i,t}}\leq \zeta\epsilon$ by induction.

    {\bf Based case: $\norm{\wb_{T_0+1} - \hat{\wb}_{i,T_0+1}}\leq \zeta \epsilon$}\\
    Based on the linear dynamics, we have
    \begin{equation*}
    \begin{split}
        \norm{\wb_{T_0+1} - \hat{\wb}_{i,T_0+1}} &= \norm{(\widehat{\Ab}_i - \Ab)\xb_{i,T_0+1} + (\widehat{\Bb}_i - \Bb)\ub_{i,T_0+1}}\\
        &= \norm{(\widehat{\Ab}_i - \Ab)\xb_{i,T_0+1} - (\widehat{\Bb}_i - \Bb)\Kb\xb_{i,T_0+1}}\\
        &\leq \epsilon(1+\kappa) \frac{2\kappa^3 W\sqrt{d_2}}{\gamma}\leq \zeta \epsilon,
    \end{split}
    \end{equation*}
    where the second equality is due to the initialization of the estimated noises and the first inequality is based on Lemma \ref{L: Bounds on states and actions during the exploration phase}.\\\\

    Suppose for a $t_0\geq T_0+1$, we have for all $t\in [T_0+1, t_0]$, $\norm{\wb_t - \hat{\wb}_{i,t}}\leq \zeta\epsilon$, then based on Equations \eqref{Eq1: Diff between estimated and real noise} and \eqref{Eq2: Diff between estimated and real noise} we get
    \begin{equation}\label{Eq3: Diff between estimated and real noise}
        \begin{aligned}
            \norm{\xb_{i,t_0+1}} &= \norm{(\Ab-\Bb\Kb)^{t_0 - T_0}\xb_{i,T_0+1} + \sum_{j=T_0+1}^{t_0} (\Ab-\Bb\Kb)^{t_0-j}(\wb_j + \Bb\Tilde{\ub}_{i,j})}\\
            &\leq \kappa^2\norm{\xb_{i,T_0+1}} + \norm{\sum_{j=T_0+1}^{t_0} (\Ab-\Bb\Kb)^{t_0-j}\wb_j} + \norm{\sum_{j=T_0+1}^{t_0}\left[ (\Ab-\Bb\Kb)^{t_0-j}\Bb\sum_{k=1}^H \Mb_{i,j}^{[k-1]} \hat{\wb}_{i,j-k}\right]}\\
            &\leq \frac{2\kappa^5 W\sqrt{d_2}}{\gamma} + \frac{\kappa^2 W}{\gamma} + \sum_{j=T_0+1}^{t_0}\left[ \norm{(\Ab-\Bb\Kb)^{t_0-j}}\norm{\Bb}\sum_{k=1}^H \norm{\Mb_{i,j}^{[k-1]}} \norm{\hat{\wb}_{i,j-k}}\right]\\
            &\leq \frac{2\kappa^5 W\sqrt{d_2}}{\gamma} + \frac{\kappa^2 W}{\gamma} + \sum_{j=T_0+1}^{t_0}\left[ \kappa^3(1-\gamma)^{t_0-j}\sum_{k=1}^H 2\kappa^3(1-\gamma)^{k-1} (W+\zeta\epsilon)\right]\\
            &\leq \frac{2\kappa^5 W\sqrt{d_2}}{\gamma} + \frac{\kappa^2 W}{\gamma} + \frac{2\kappa^6(W+\zeta\epsilon)}{\gamma^2}
        \end{aligned}
    \end{equation}
    and 
    \begin{equation}\label{Eq4: Diff between estimated and real noise}
    \begin{aligned}
        \norm{\wb_{t_0+1} - \hat{\wb}_{i,t_0+1}} &= \norm{(\widehat{\Ab}_i - \Ab)\xb_{i,t_0+1} + (\widehat{\Bb}_i - \Bb)\ub_{i,t_0+1}}\\
        &=\norm{(\widehat{\Ab}_i - \Ab)\xb_{i,t_0+1} + (\widehat{\Bb}_i - \Bb)(-\Kb\xb_{i,t_0+1} + \sum_{j=1}^H \Mb_{i,t}^{[j-1]} \hat{\wb}_{i,t_0+1-j})}\\
        &\leq \epsilon(1+\kappa)\big(\frac{2\kappa^5 W\sqrt{d_2}}{\gamma} + \frac{\kappa^2 W}{\gamma} + \frac{2\kappa^6(W+\zeta\epsilon)}{\gamma^2}\big) + \epsilon2\kappa^3 (W+\zeta\epsilon) \sum_{j=1}^H(1-\gamma)^{j-1}\\
        &\leq \epsilon(1+\kappa)\big(\frac{2\kappa^5 W\sqrt{d_2}}{\gamma} + \frac{\kappa^2 W}{\gamma} + \frac{2\kappa^6(W+\zeta\epsilon)}{\gamma^2}\big) + \frac{\epsilon2\kappa^3 (W+\zeta\epsilon)}{\gamma}\\
        &\leq \frac{12 W\sqrt{d_2}\kappa^7}{\gamma^2}\epsilon + \frac{6\kappa^7}{\gamma^2}\zeta \epsilon^2\leq \frac{24 W\sqrt{d_2}\kappa^7}{\gamma^2}\epsilon = \zeta\epsilon.
    \end{aligned}        
    \end{equation}
    Based on the rule of induction, the result is proved.
\end{proof}

\subsection{Bounds for The Distributed Online Learning with Memory and the Individual Regret (Unknown System)}
For the unknown case, as the local surrogate function is built on each agent's system estimate, the analysis of distributed OL is not directly applicable. To tackle this issue, in Lemma \ref{L: Difference between surrogate states (actions) on the real and estiamted systems} we quantify the deviation of the evolvement of the surrogate state due to the effect of different estimates. With Lemma \ref{L: Difference between surrogate states (actions) on the real and estiamted systems}, we provide the regret bound on the corresponding distributed OL with memory of the unknown case in Theorem \ref{T: Regret bound for OCO with memory (Unknown case)} and show the proof of Theorem \ref{T: Regret bound (unknown case)}.

\begin{lemma}\label{L: Difference between surrogate states (actions) on the real and estiamted systems}
Suppose the conditions of Lemma \ref{L: Diff between estimated and real noise} hold and $\norm{\Ab - \widehat{\Ab}_i},\norm{\Bb - \widehat{\Bb}_i}\leq \epsilon\text{ and }\norm{\widehat{\Ab}_i - \widehat{\Ab}_j},\norm{\widehat{\Bb}_i - \widehat{\Bb}_j}\leq \epsilon'$. For any $\Mb\in \mathcal{M}=\{\Mb: \Mb^{[i-1]}\leq C(1-\gamma)^i\}$, we have $\forall i,l \in [m]$
\begin{equation*}
\begin{split}
    &\norm{\yb^{\Kb}_{t+1}(\Mb|\widehat{\Ab}_i, \widehat{\Bb}_i,\{\hat{\wb}_i\}) - \yb^{\Kb}_{t+1}(\Mb|\Ab, \Bb,\{\wb\})}\\
    \leq &\frac{2\kappa^5\epsilon}{(\gamma - 2\kappa^3\epsilon)^2} + \frac{2\kappa^5H^2C(\kappa+\epsilon)\epsilon}{(1-\gamma+2\kappa^3\epsilon)(\gamma - 2\kappa^3\epsilon)} + \frac{(H+1)\kappa^2C\epsilon}{(1-\gamma)\gamma }\\
    + & \left(\frac{\kappa^2 + (H)\kappa^3C}{\gamma}\right)(2W + \zeta\epsilon) \text{ for }t+1\in[T_0+1,T_0+2H+1],
\end{split}
\end{equation*}
and
\begin{equation*}
\begin{split}
    &\norm{\yb^{\Kb}_{t+1}(\Mb|\widehat{\Ab}_i, \widehat{\Bb}_i,\{\hat{\wb}_i\}) - \yb^{\Kb}_{t+1}(\Mb|\Ab, \Bb,\{\wb\})}\\
    \leq &\frac{2\kappa^5\epsilon}{(\gamma - 2\kappa^3\epsilon)^2} + \frac{2\kappa^5H^2C(\kappa+\epsilon)\epsilon}{(1-\gamma+2\kappa^3\epsilon)(\gamma - 2\kappa^3\epsilon)} + \frac{(H+1)\kappa^2C\epsilon}{(1-\gamma)\gamma }\\
    + & \left(\frac{\kappa^2 + (H)\kappa^3C}{\gamma}\right)(\zeta\epsilon)\text{ for }t+1\geq T_0+2H+2,
\end{split}
\end{equation*}
and 
\begin{equation*}
\begin{split}
    &\norm{\yb^{\Kb}_{t+1}(\Mb|\widehat{\Ab}_i, \widehat{\Bb}_i,\{\hat{\wb}_i\}) - \yb^{\Kb}_{t+1}(\Mb|\widehat{\Ab}_l, \widehat{\Bb}_l,\{\hat{\wb}_l\})}\\
    \leq &\frac{2\kappa^5\epsilon^{\prime}}{(\gamma - 2\kappa^3\epsilon)^2} + \frac{2\kappa^5H^2C(\kappa+\epsilon)\epsilon^{\prime}}{(1-\gamma+2\kappa^3\epsilon)(\gamma - 2\kappa^3\epsilon)} + \frac{(H+1)\kappa^2C\epsilon^{\prime}}{(1-\gamma+2\kappa^3\epsilon)(\gamma - 2\kappa^3\epsilon)}\\
    + & \left(\frac{\kappa^2 + (H)(\kappa+\epsilon)\kappa^2C}{\gamma - 2\kappa^3\epsilon}\right)(2\zeta\epsilon).
\end{split}    
\end{equation*}
\end{lemma}
\begin{proof}
    Based on the expression of the surrogate state, we have
    \begin{equation}\label{Eq1: Difference between surrogate states (actions) on the real and estiamted systems}
    \resizebox{\linewidth}{!}
    {
    \begin{minipage}{\linewidth}
    $
    \begin{aligned}
        &\norm{\yb^{\Kb}_{t+1}(\Mb|\widehat{\Ab}_i, \widehat{\Bb}_i,\{\hat{\wb}_i\}) - \yb^{\Kb}_{t+1}(\Mb|\Ab, \Bb,\{\wb\})}\\
        = &\norm{\sum_{j=0}^{2H} \left[\Psi^{\Kb,H}_{t,j}(\Mb|\widehat{\Ab}_i, \widehat{\Bb}_i) \hat{\wb}_{i,t-j} - \Psi^{\Kb,H}_{t,j}(\Mb|\Ab,\Bb) \wb_{t-j}\right]}\\
        = &\norm{\sum_{j=0}^{2H} \left[\Psi^{\Kb,H}_{t,j}(\Mb|\widehat{\Ab}_i, \widehat{\Bb}_i) \hat{\wb}_{i,t-j} - \Psi^{\Kb,H}_{t,j}(\Mb|\Ab,\Bb)\hat{\wb}_{i,t-j} + \Psi^{\Kb,H}_{t,j}(\Mb|\Ab,\Bb)\hat{\wb}_{i,t-j} - \Psi^{\Kb,H}_{t,j}(\Mb|\Ab,\Bb) \wb_{t-j}\right]}\\
        \leq &\sum_{j=0}^{2H} \left[\norm{\big(\Psi^{\Kb,H}_{t,j}(\Mb|\widehat{\Ab}_i, \widehat{\Bb}_i) - \Psi^{\Kb,H}_{t,j}(\Mb|\Ab,\Bb)\big)\hat{\wb}_{i,t-j}} + \norm{\Psi^{\Kb,H}_{t,j}(\Mb|\Ab,\Bb)}\norm{\hat{\wb}_{i,t-j} - \wb_{t-j}}\right].
    \end{aligned}
    $    
    \end{minipage}
    }
    \end{equation}
    From the definition of the disturbance-state transfer matrix, we get
    \begin{equation}\label{Eq2: Difference between surrogate states (actions) on the real and estiamted systems}
    \resizebox{\linewidth}{!}
    {
    $
    \begin{aligned}
        &\norm{\Psi^{\Kb,H}_{t,j}(\Mb|\widehat{\Ab}_i, \widehat{\Bb}_i) - \Psi^{\Kb,H}_{t,j}(\Mb|\Ab,\Bb)}\\
        =&\norm{(\widehat{\Ab}_i - \widehat{\Bb}_i\Kb)^j \mathbb{1}_{j\leq H} + \sum_{k=0}^H (\widehat{\Ab}_i - \widehat{\Bb}_i\Kb)^k\widehat{\Bb}_i\Mb^{[j-k-1]}\mathbb{1}_{j-k\in [1,H]} - (\Ab - \Bb\Kb)^j \mathbb{1}_{j\leq H} - \sum_{k=0}^H (\Ab - \Bb\Kb)^k\Bb\Mb^{[j-k-1]}\mathbb{1}_{j-k\in [1,H]}}\\
        \leq &\norm{(\widehat{\Ab}_i - \widehat{\Bb}_i\Kb)^j - (\Ab - \Bb\Kb)^j}
        +\norm{\sum_{k=0}^H \left[(\widehat{\Ab}_i - \widehat{\Bb}_i\Kb)^k\widehat{\Bb}_i\Mb^{[j-k-1]} - (\Ab - \Bb\Kb)^k\widehat{\Bb}_i\Mb^{[j-k-1]}\right]}\\
        + &\norm{\sum_{k=0}^H \left[(\Ab - \Bb\Kb)^k\widehat{\Bb}_i\Mb^{[j-k-1]} - (\Ab - \Bb\Kb)^k\Bb\Mb^{[j-k-1]}\right]}\\
        \leq & \kappa^2j(1-\gamma+2\kappa^3\epsilon)^{j-1}(2\kappa^3\epsilon) + \sum_{k=0}^H \kappa^2k(1-\gamma+2\kappa^3\epsilon)^{k-1}(2\kappa^3\epsilon)\norm{\hat{\Bb}_i}C(1-\gamma)^{j-k-1} + (H+1)\epsilon\kappa^2C(1-\gamma)^{j-1},
    \end{aligned}
    $    
    }
    \end{equation}
    where the last inequality is by applying Lemma \ref{L: Series of moment differences} on the fact that from Lemma \ref{L:strong stability for a close system}, we have $(\Ab - \Bb\Kb) = \Hb\Lb_1\Hb^{-1}$ and $(\widehat{\Ab}_i - \widehat{\Bb}_i\Kb) = \Hb\Lb_2\Hb^{-1}$ such that $\norm{\Lb_1}, \norm{\Lb_2}\leq 1-\gamma+2\kappa^3\epsilon$ and $\norm{\Lb_1 - \Lb_2}\leq 2\kappa^3\epsilon$.
    \\\\
    For the term $\sum_{j=0}^{2H} \left[\norm{\Psi^{\Kb,H}_{t,j}(\Mb|\Ab,\Bb)}\norm{\hat{\wb}_{i,t-j} - \wb_{t-j}}\right]$, we consider two cases:
    \begin{itemize}
        \item $t+1\in[T_0+1,T_0+2H+1]$:
        \begin{equation}\label{Eq3: Difference between surrogate states (actions) on the real and estiamted systems}
        \begin{split}
            &\sum_{j=0}^{2H} \left[\norm{\Psi^{\Kb,H}_{t,j}(\Mb|\Ab,\Bb)}\norm{\hat{\wb}_{i,t-j} - \wb_{t-j}}\right]\\
            \leq &\sum_{j=0}^{2H} \left[\big(\kappa^2(1-\gamma)^j + H \kappa^3C(1-\gamma)^{j-1}\big)(\norm{\hat{\wb}_{i,t-j}} + \norm{\wb_{t-j}})\right]\\
            \leq &\left(\frac{\kappa^2 + (H)\kappa^3C}{\gamma}\right)(2W + \zeta\epsilon).
        \end{split}    
        \end{equation}
        \item $t+1\geq T_0+2H+2$:
        \begin{equation}\label{Eq4: Difference between surrogate states (actions) on the real and estiamted systems}
        \begin{split}
            &\sum_{j=0}^{2H} \left[\norm{\Psi^{\Kb,H}_{t,j}(\Mb|\Ab,\Bb)}\norm{\hat{\wb}_{i,t-j} - \wb_{t-j}}\right]\\
            \leq &\sum_{j=0}^{2H} \left[\big(\kappa^2(1-\gamma)^j + H \kappa^3C(1-\gamma)^{j-1}\big)(\zeta\epsilon)\right]\\
            \leq &\left(\frac{\kappa^2 + (H)\kappa^3C}{\gamma}\right)(\zeta\epsilon).
        \end{split}    
        \end{equation}
    \end{itemize}
    Substituting Equations \eqref{Eq2: Difference between surrogate states (actions) on the real and estiamted systems}, \eqref{Eq3: Difference between surrogate states (actions) on the real and estiamted systems} and \eqref{Eq4: Difference between surrogate states (actions) on the real and estiamted systems} into Equation \eqref{Eq1: Difference between surrogate states (actions) on the real and estiamted systems}, we get
    \begin{equation}\label{Eq5: Difference between surrogate states (actions) on the real and estiamted systems}
    \begin{split}
        &\norm{\yb^{\Kb}_{t+1}(\Mb|\widehat{\Ab}_i, \widehat{\Bb}_i,\{\hat{\wb}_i\}) - \yb^{\Kb}_{t+1}(\Mb|\Ab, \Bb,\{\wb\})}\\
        \leq &\left[\frac{2\kappa^5\epsilon}{(\gamma - 2\kappa^3\epsilon)^2} + \frac{2\kappa^5H^2C(\kappa+\epsilon)\epsilon}{(1-\gamma+2\kappa^3\epsilon)(\gamma - 2\kappa^3\epsilon)} + \frac{(H+1)\kappa^2C\epsilon}{(1-\gamma)\gamma }\right](W +\zeta\epsilon)\\
        + & \left(\frac{\kappa^2 + (H)\kappa^3C}{\gamma}\right)(2W + \zeta\epsilon) \text{ for }t+1\in[T_0+1,T_0+2H+1],
    \end{split}
    \end{equation}
    and
    \begin{equation}\label{Eq6: Difference between surrogate states (actions) on the real and estiamted systems}
    \begin{split}
        &\norm{\yb^{\Kb}_{t+1}(\Mb|\widehat{\Ab}_i, \widehat{\Bb}_i,\{\hat{\wb}_i\}) - \yb^{\Kb}_{t+1}(\Mb|\Ab, \Bb,\{\wb\})}\\
        \leq &\left[\frac{2\kappa^5\epsilon}{(\gamma - 2\kappa^3\epsilon)^2} + \frac{2\kappa^5H^2C(\kappa+\epsilon)\epsilon}{(1-\gamma+2\kappa^3\epsilon)(\gamma - 2\kappa^3\epsilon)} + \frac{(H+1)\kappa^2C\epsilon}{(1-\gamma)\gamma }\right](W+\zeta\epsilon)\\
        + & \left(\frac{\kappa^2 + (H)\kappa^3C}{\gamma}\right)(\zeta\epsilon)\text{ for }t+1\geq T_0+2H+2.
    \end{split}
    \end{equation}

    Knowing that from Lemma \ref{L:strong stability for a close system}, we have $(\widehat{\Ab}_i - \widehat{\Bb}_i\Kb) = \Hb\Lb_1\Hb^{-1}$ and $(\widehat{\Ab}_{l} - \widehat{\Bb}_{l}\Kb) = \Hb\Lb_2\Hb^{-1}$ such that $\norm{\Lb_1}, \norm{\Lb_2}\leq 1-\gamma+2\kappa^3\epsilon$ and $\norm{\Lb_1 - \Lb_2}\leq 2\kappa^3\epsilon^{\prime}$, as Equation \eqref{Eq2: Difference between surrogate states (actions) on the real and estiamted systems}, we have
    \begin{equation}\label{Eq7: Difference between surrogate states (actions) on the real and estiamted systems}
    \begin{split}
        &\norm{\Psi^{\Kb,H}_{t,j}(\Mb|\widehat{\Ab}_i, \widehat{\Bb}_i) - \Psi^{\Kb,H}_{t,j}(\Mb|\widehat{\Ab}_l, \widehat{\Bb}_l)}\\
        \leq & \kappa^2j(1-\gamma+2\kappa^3\epsilon)^{j-1}(2\kappa^3\epsilon^{\prime}) + \sum_{k=0}^H \kappa^2k(1-\gamma+2\kappa^3\epsilon)^{k-1}(2\kappa^3\epsilon^{\prime})\norm{\widehat{\Bb}_i}C(1-\gamma)^{j-k-1}\\
        + &(H+1)\epsilon^{\prime}\kappa^2C(1-\gamma+2\kappa^3\epsilon)^{j-1}.
    \end{split}
    \end{equation}
    Based on the definition of the estimated noise sequence, we have $\norm{\hat{\wb}_{i,t} - \hat{\wb}_{l,t}}\leq 2\zeta\epsilon$ for all $t$. Therefore, for $t\geq T_0+1$ we get
    \begin{equation}\label{Eq8: Difference between surrogate states (actions) on the real and estiamted systems}
    \begin{split}
        &\sum_{j=0}^{2H} \left[\norm{\Psi^{\Kb,H}_{t,j}(\Mb|\widehat{\Ab}_l,\widehat{\Bb}_l)}\norm{\hat{\wb}_{i,t-j} - \hat{\wb}_{l,t-j}}\right]\\
        \leq &\sum_{j=0}^{2H} \left[\big(\kappa^2(1-\gamma+2\kappa^3\epsilon)^j + H (\kappa+\epsilon)\kappa^2C(1-\gamma+2\kappa^3\epsilon)^{j-1}\big)(2\zeta\epsilon)\right]\\
        \leq &\left(\frac{\kappa^2 + (H)(\kappa+\epsilon)\kappa^2C}{\gamma - 2\kappa^3\epsilon}\right)(2\zeta\epsilon).
    \end{split}
    \end{equation}    
    Based on Equations \eqref{Eq7: Difference between surrogate states (actions) on the real and estiamted systems} and \eqref{Eq8: Difference between surrogate states (actions) on the real and estiamted systems}, we can get
    \begin{equation}\label{Eq9: Difference between surrogate states (actions) on the real and estiamted systems}
    \begin{split}
        &\norm{\yb^{\Kb}_{t+1}(\Mb|\widehat{\Ab}_i, \widehat{\Bb}_i,\{\hat{\wb}_i\}) - \yb^{\Kb}_{t+1}(\Mb|\widehat{\Ab}_l, \widehat{\Bb}_l,\{\hat{\wb}_l\})}\\
        \leq &\left[\frac{2\kappa^5\epsilon^{\prime}}{(\gamma - 2\kappa^3\epsilon)^2} + \frac{2\kappa^5H^2C(\kappa+\epsilon)\epsilon^{\prime}}{(1-\gamma+2\kappa^3\epsilon)(\gamma - 2\kappa^3\epsilon)} + \frac{(H+1)\kappa^2C\epsilon^{\prime}}{(1-\gamma+2\kappa^2\epsilon)(\gamma - 2\kappa^2\epsilon)}\right](W+\zeta\epsilon)\\
        + & \left(\frac{\kappa^2 + (H)(\kappa+\epsilon)\kappa^2C}{\gamma - 2\kappa^3\epsilon}\right)(2\zeta\epsilon).
    \end{split}    
    \end{equation}
    
\end{proof}

\begin{lemma}\label{L: Diff between resulting states based on different initial states}
    For a linear system evolving based on $\{\Ab,\Bb,\{\wb\}\}$ and the control is chosen as the form: $\ub_t = -\Kb\xb_t + \Tilde{\ub}_t$, where $\Tilde{\ub}_t$ does not depend on $\xb_t$. For two state sequences $\{\xb\}$ and $\{\xb^{\prime}\}$ starting from two different initial states: $\xb_{t_0}$ and $\xb_{t_0} + \Delta\xb$, we have
    \begin{equation*}
    \begin{split}
        \norm{\xb_{t} - \xb^{\prime}_t} &\leq \norm{(\Ab-\Bb\Kb)^{t-t_0}}\norm{\Delta \xb},\\
        \norm{\ub_{t} - \ub^{\prime}_t} &\leq \kappa\norm{(\Ab-\Bb\Kb)^{t-t_0}}\norm{\Delta \xb}.
    \end{split}
    \end{equation*}
\end{lemma}

\begin{lemma}\label{L: Consecutive distance (unknown case)}
    Let Algorithm 2 run with step size $\eta>0$ and define $\Mb_t = \frac{1}{m}\sum_{i=1}^m\Mb_{i,t}$. Under Assumptions \ref{A1} to \ref{A4}, we have that $\forall i \in [m]$    
    \begin{equation*}
        \norm{\Mb_{t} - \Mb_{i,t}}_F\leq \frac{2\eta G_c D d\sqrt{d}\sqrt{H} \big(\frac{\kappa^3(\kappa+\epsilon) (W+\zeta\epsilon)}{\gamma - 2\kappa^3\epsilon} + (W+\zeta\epsilon)\big)\sqrt{m}}{1-\beta}
    \end{equation*}
    and
    \begin{equation*}
        \norm{\Mb_t - \widehat{\Mb}_{i,t+1}}_F\frac{2\eta G_c D d\sqrt{d}\sqrt{H} \big(\frac{\kappa^3(\kappa+\epsilon) (W+\zeta\epsilon)}{\gamma - 2\kappa^3\epsilon} + (W+\zeta\epsilon)\big)\sqrt{m}}{1-\beta},
    \end{equation*}
    where $D = \frac{(W+\zeta\epsilon)(\kappa^2 + 2H(\kappa+\epsilon)\kappa^5)}{(\gamma - 2\kappa^3\epsilon)(1-\kappa^2(1-\gamma+2\kappa^3\epsilon)^{(H+1)})} + \frac{2(W+\zeta\epsilon)\kappa^3}{\gamma}$.
    \end{lemma}
    \begin{proof}
        The proof follows the same process as Lemma \ref{L: Consecutive distance}. The main difference is that since ideal cost of agent $i$ is now defined as $f_{i,t}(\Mb_{i,t}|\widehat{\Ab}_i, \widehat{\Bb}_i, \{\hat{\wb}_i\})$. Knowing that $\forall i,\; \norm{\widehat{\Ab}_i - \Ab}, \norm{\widehat{\Bb}_i - \Bb}\leq \epsilon$, based on Lemma \ref{L:strong stability for a close system}, $\Kb$ is shown to be $(\kappa, \gamma-2\kappa^3\epsilon)$-strongly stable w.r.t. $(\widehat{\Ab}_i, \widehat{\Bb}_i)$. Also from Lemma \ref{L: Diff between estimated and real noise}, we have $\forall i,t,\;\norm{\hat{\wb}_{i,t}}\leq W + \zeta\epsilon$, where $\zeta = \frac{24 W\sqrt{d_2}\kappa^7}{\gamma^2}$. Based on these facts and Lemma \ref{L: Gradient bound for reparameterized functions}, we have
        \begin{equation}\label{Eq1: Consecutive distance (unknown case)}
            \norm{\nabla_{\Mb}f_{i,t}(\Mb_{i,t}|\widehat{\Ab}_i, \widehat{\Bb}_i, \{\hat{\wb}_i\})}_F\leq G_c D d\sqrt{d}\sqrt{H} \big(\frac{\kappa^3(\kappa+\epsilon) (W+\zeta\epsilon)}{\gamma - 2\kappa^3\epsilon} + (W+\zeta\epsilon)\big),
        \end{equation}
        where $D = \frac{(W+\zeta\epsilon)(\kappa^2 + 2H(\kappa+\epsilon)\kappa^5)}{(\gamma - 2\kappa^3\epsilon)(1-\kappa^2(1-\gamma+2\kappa^3\epsilon)^{(H+1)})} + \frac{2(W+\zeta\epsilon)\kappa^3}{\gamma}$.
        Based on Equation \eqref{Eq1: Consecutive distance (unknown case)} and the same proof procedure as Lemma \ref{L: Consecutive distance}, the result is proved.
    \end{proof}

\begin{theorem}\label{T: Regret bound for OCO with memory (Unknown case)}(Regret bound for OCO with memory)
    Let Algorithm 2 run with step size $\eta>0$. Under Assumptions \ref{A1} to \ref{A4}, we have that $\forall j \in [m]$   
    \begin{equation*}
    \resizebox{\linewidth}{!}
    {$
    \begin{aligned}
        &\sum_{t=T_0+1}^T \sum_i f_{i,t}(\Mb_{j,t-H-1:t}|\widehat{\Ab}_j, \widehat{\Bb}_j, \{\hat{\wb}_j\}) - f_{i,t}(\Mb^*|\widehat{\Ab}_i, \widehat{\Bb}_i, \{\hat{\wb}_i\})\\
        \leq &O\bigg(\eta Tm H^3(G_c D)^2(\kappa+\epsilon)\kappa^3 (W+\zeta\epsilon)\frac{ \big(\frac{\kappa^3(\kappa+\epsilon) (W+\zeta\epsilon)}{\gamma - 2\kappa^3\epsilon} + (W+\zeta\epsilon)\big)\sqrt{m}}{1-\beta}\bigg)\\
        + &O\bigg(mTG_cD \sqrt{2}\left[\kappa\left(\frac{2\kappa^5\epsilon^{\prime}(W+\zeta\epsilon)}{(\gamma - 2\kappa^3\epsilon)^2} + \frac{4\kappa^8H^2(\kappa+\epsilon)\epsilon^{\prime}(W+\zeta\epsilon)}{(1-\gamma+2\kappa^3\epsilon)(\gamma - 2\kappa^3\epsilon)} + \frac{(H+1)2\kappa^5\epsilon^{\prime}(W+\zeta\epsilon)}{(1-\gamma+2\kappa^3\epsilon)(\gamma - 2\kappa^3\epsilon)} + \left(\frac{\kappa^2 + (H)(\kappa+\epsilon)2\kappa^5}{\gamma - 2\kappa^3\epsilon}\right)(2\zeta\epsilon)\right) + \frac{4\kappa^3\zeta\epsilon}{\gamma}\right]\bigg)\\
        + &O\big(\eta T m\sqrt{m}G_cDd^3H\frac{\big(\frac{\kappa^3(\kappa+\epsilon) (W+\zeta\epsilon)}{\gamma - 2\kappa^3\epsilon} + (W+\zeta\epsilon)\big)^2}{1-\beta}\big) + O\bigg(\eta Tm \big(G_c D d\sqrt{d}\sqrt{H} \big(\frac{\kappa^3(\kappa+\epsilon) (W+\zeta\epsilon)}{\gamma - 2\kappa^3\epsilon} + (W+\zeta\epsilon)\big)\big)^2\bigg) + O(\frac{1}{\eta})\\
        + &O\bigg(\frac{1}{\eta}Tm \big(\frac{\eta G_c D d\sqrt{d}\sqrt{H} \big(\frac{\kappa^3(\kappa+\epsilon) (W+\zeta\epsilon)}{\gamma - 2\kappa^3\epsilon} + (W+\zeta\epsilon)\big)\sqrt{m}}{1-\beta}\big)^2\bigg).
    \end{aligned}
    $}
    \end{equation*}
    where $D = \frac{(W+\zeta\epsilon)(\kappa^2 + 2H(\kappa+\epsilon)\kappa^5)}{(\gamma - 2\kappa^3\epsilon)(1-\kappa^2(1-\gamma+2\kappa^3\epsilon)^{(H+1)})} + \frac{2(W+\zeta\epsilon)\kappa^3}{\gamma}$.
    \end{theorem}
    \begin{proof}
    Like Equation \eqref{Eq1: Regret bound for OCO with memory}, based on the update rule,  $\forall i\in [m]$ and $t\in[T_0+1,\ldots, T]$ we have
    \begin{equation}\label{Eq1: Regret bound for OCO with memory (Unknown case)}
    \begin{split}
        &f_{i,t}(\Mb_{i,t}|\widehat{\Ab}_i, \widehat{\Bb}_i, \{\hat{\wb}_i\}) - f_{i,t}(\Mb^*|\widehat{\Ab}_i, \widehat{\Bb}_i, \{\hat{\wb}_i\})\\
    \leq &\frac{1}{2}\eta\norm{\nabla f_{i,t}(\Mb_{i,t}|\widehat{\Ab}_i, \widehat{\Bb}_i, \{\hat{\wb}_i\})}^2 + \frac{1}{2\eta}\norm{\Mb_{i,t}-\Mb^*}^2 - \frac{1}{2\eta}\norm{\Mb_{i,t+1} - \Mb^*}^2\\
    - &\frac{1}{\eta}\big(\Mb_{i,t} - \sum_j (\Pb_{ji}\Mb_{j.t})\big)^{\top}(\widehat{\Mb}_{i,t+1} - \widehat{\Mb}_{t+1}) - \frac{1}{\eta}\big(\Mb_{i,t} - \sum_j (\Pb_{ji}\Mb_{j.t})\big)^{\top}(\widehat{\Mb}_{t+1} - \Mb^*).
    \end{split}    
    \end{equation}    
    As Equation \eqref{Eq2: Regret bound for OCO with memory}, we get
    \begin{equation}\label{Eq2: Regret bound for OCO with memory (Unknown case)}
    \resizebox{\linewidth}{!}
    {
    \begin{minipage}{\linewidth}
    $
    \begin{aligned}
        &f_{i,t}(\Mb_{j,t}|\widehat{\Ab}_i, \widehat{\Bb}_i, \{\hat{\wb}_i\}) - f_{i,t}(\Mb^*|\widehat{\Ab}_i, \widehat{\Bb}_i, \{\hat{\wb}_i\})\\
        =& f_{i,t}(\Mb_{j,t}|\widehat{\Ab}_i, \widehat{\Bb}_i, \{\hat{\wb}_i\}) - f_{i,t}(\Mb_{i,t}|\widehat{\Ab}_i, \widehat{\Bb}_i, \{\hat{\wb}_i\}) + f_{i,t}(\Mb_{i,t}|\widehat{\Ab}_i, \widehat{\Bb}_i, \{\hat{\wb}_i\}) - f_{i,t}(\Mb^*|\widehat{\Ab}_i, \widehat{\Bb}_i, \{\hat{\wb}_i\})\\
        \leq & G_c D d\sqrt{d}\sqrt{H} \big(\frac{\kappa^3(\kappa+\epsilon) (W+\zeta\epsilon)}{\gamma - 2\kappa^3\epsilon} + (W+\zeta\epsilon)\big)\frac{2\eta G_c D d\sqrt{d}\sqrt{H} \big(\frac{\kappa^3(\kappa+\epsilon) (W+\zeta\epsilon)}{\gamma - 2\kappa^3\epsilon} + (W+\zeta\epsilon)\big)\sqrt{m}}{1-\beta}\\
        + &\frac{1}{2}\eta\norm{\nabla f_{i,t}(\Mb_{i,t}|\widehat{\Ab}_i, \widehat{\Bb}_i, \{\hat{\wb}_i\})}^2 + \frac{1}{2\eta}\norm{\Mb_{i,t}-\Mb^*}^2 - \frac{1}{2\eta}\norm{\Mb_{i,t+1} - \Mb^*}^2\\
    - &\frac{1}{\eta}\big(\Mb_{i,t} - \sum_j (\Pb_{ji}\Mb_{j.t})\big)^{\top}(\widehat{\Mb}_{i,t+1} - \widehat{\Mb}_{t+1}) - \frac{1}{\eta}\big(\Mb_{i,t} - \sum_j (\Pb_{ji}\Mb_{j.t})\big)^{\top}(\widehat{\Mb}_{t+1} - \Mb^*).
    \end{aligned}
    $    
    \end{minipage}
    }    
    \end{equation}
    Summing Equation \eqref{Eq2: Regret bound for OCO with memory (Unknown case)} over $i$ and $t$, we have 
    \begin{equation}\label{Eq3: Regret bound for OCO with memory (Unknown case)}
    \resizebox{\linewidth}{!}
    {
    \begin{minipage}{\linewidth}
    $
    \begin{aligned}
        &\sum_{t=T_0+1}^T \sum_i f_{i,t}(\Mb_{j,t}|\widehat{\Ab}_i, \widehat{\Bb}_i, \{\hat{\wb}_i\}) - f_{i,t}(\Mb^*|\widehat{\Ab}_i, \widehat{\Bb}_i, \{\hat{\wb}_i\})\\
        \leq & 2\eta T m\sqrt{m}G_cDd^3H\frac{\big(\frac{\kappa^3(\kappa+\epsilon) (W+\zeta\epsilon)}{\gamma - 2\kappa^3\epsilon} + (W+\zeta\epsilon)\big)^2}{1-\beta}        
        + \frac{1}{2}\eta\sum_{t=T_0+1}^T\sum_i\norm{\nabla f_{i,t}(\Mb_{i,t}|\widehat{\Ab}_i, \widehat{\Bb}_i, \{\hat{\wb}_i\})}^2\\      
        + &\frac{1}{\eta}\sum_i\left(\Mb_{i,T+1}^{\top} - \Mb_{i,T_0+1}^{\top}\right)\Mb^*
        + \frac{1}{\eta}\sum_{t=T_0+1}^T\sum_i \norm{\big(\Mb_{i,t} - \sum_j (\Pb_{ji}\Mb_{j.t})\big)}\norm{(\widehat{\Mb}_{i,t+1} - \widehat{\Mb}_{t+1})}\\
        + &\frac{1}{2\eta}\sum_i\norm{\Mb_{i,T_0+1}}^2.
    \end{aligned}
    $    
    \end{minipage}
    }
    \end{equation}

    Based on Lemma \ref{L: Difference between surrogate states (actions) on the real and estiamted systems}, we can get
    \begin{equation}\label{Eq4: Regret bound for OCO with memory (Unknown case)}
    \resizebox{\linewidth}{!}
    {$
    \begin{aligned}
        &f_{i,t}(\Mb_{j,t}|\widehat{\Ab}_j, \widehat{\Bb}_j, \{\hat{\wb}_j\}) - f_{i,t}(\Mb_{j,t}|\widehat{\Ab}_i, \widehat{\Bb}_i, \{\hat{\wb}_i\})\\
        =& c_{i,t}\big(\yb^{\Kb}_{t}(\Mb_{j,t}|\widehat{\Ab}_j, \widehat{\Bb}_j,\{\hat{\wb}_j\}), \vb^{\Kb}_{t}(\Mb_{j,t}|\widehat{\Ab}_j, \widehat{\Bb}_j,\{\hat{\wb}_j\})\big) - c_{i,t}\big(\yb^{\Kb}_{t}(\Mb_{j,t}|\widehat{\Ab}_i, \widehat{\Bb}_i,\{\hat{\wb}_i\}), \vb^{\Kb}_{t}(\Mb_{j,t}|\widehat{\Ab}_i, \widehat{\Bb}_i,\{\hat{\wb}_i\})\big)\\
        \leq & G_c D \sqrt{\norm{\yb^{\Kb}_{t}(\Mb_{j,t}|\widehat{\Ab}_j, \widehat{\Bb}_j,\{\hat{\wb}_j\}) - \yb^{\Kb}_{t}(\Mb_{j,t}|\widehat{\Ab}_i, \widehat{\Bb}_i,\{\hat{\wb}_i\})}^2 + \norm{\vb^{\Kb}_{t}(\Mb_{j,t}|\widehat{\Ab}_j, \widehat{\Bb}_j,\{\hat{\wb}_j\}) - \vb^{\Kb}_{t}(\Mb_{j,t}|\widehat{\Ab}_i, \widehat{\Bb}_i,\{\hat{\wb}_i\})}^2}\\
        \leq & G_cD \sqrt{2}\left[\kappa\left(\frac{2\kappa^5\epsilon^{\prime}(W+\zeta\epsilon)}{(\gamma - 2\kappa^3\epsilon)^2} + \frac{4\kappa^8H^2(\kappa+\epsilon)\epsilon^{\prime}(W+\zeta\epsilon)}{(1-\gamma+2\kappa^3\epsilon)(\gamma - 2\kappa^3\epsilon)} + \frac{(H+1)2\kappa^5\epsilon^{\prime}(W+\zeta\epsilon)}{(1-\gamma+2\kappa^3\epsilon)(\gamma - 2\kappa^3\epsilon)} + \left(\frac{\kappa^2 + (H)(\kappa+\epsilon)2\kappa^5}{\gamma - 2\kappa^3\epsilon}\right)(2\zeta\epsilon)\right) + \frac{4\kappa^3\zeta\epsilon}{\gamma}\right],
    \end{aligned}
    $}    
    \end{equation}
    where the last inequality is due to the fact that $\norm{\hat{\wb}_{j,t} - \hat{\wb}_{i,t}}\leq 2\zeta\epsilon$ and 
    \begin{equation*}
    \begin{split}
        &\vb^{\Kb}_{t}(\Mb_{j,t}|\widehat{\Ab}_j, \widehat{\Bb}_j,\{\hat{\wb}_j\}) - \vb^{\Kb}_{t}(\Mb_{j,t}|\widehat{\Ab}_i, \widehat{\Bb}_i,\{\hat{\wb}_i\})\\
        = &-\Kb\big(\yb^{\Kb}_{t}(\Mb_{j,t}|\widehat{\Ab}_j, \widehat{\Bb}_j,\{\hat{\wb}_j\}) - \yb^{\Kb}_{t}(\Mb_{j,t}|\widehat{\Ab}_i, \widehat{\Bb}_i,\{\hat{\wb}_i\})\big) + \sum_{k=1}^H \Mb_{j,t}^{[k-1]}(\hat{\wb}_{j,t-k} - \hat{\wb}_{i,t-k}).
    \end{split}    
    \end{equation*}
    As Equation \eqref{Eq5: Regret bound for OCO with memory}, we get
    \begin{equation}\label{Eq5: Regret bound for OCO with memory (Unknown case)}
    \begin{split}
        &f_{i,t}(\Mb_{j,t-H-1:t}|\widehat{\Ab}_j, \widehat{\Bb}_j, \{\hat{\wb}_j\}) - f_{i,t}(\Mb_{j,t}|\widehat{\Ab}_j, \widehat{\Bb}_j, \{\hat{\wb}_j\})\\
        \leq &O\bigg(\eta H^3(G_c D)^2(\kappa+\epsilon)\kappa^3 (W+\zeta\epsilon)\frac{ \big(\frac{\kappa^3(\kappa+\epsilon) (W+\zeta\epsilon)}{\gamma - 2\kappa^3\epsilon} + (W+\zeta\epsilon)\big)\sqrt{m}}{1-\beta}\bigg).
    \end{split}    
    \end{equation}
    Summing Equations \eqref{Eq4: Regret bound for OCO with memory (Unknown case)} \& \eqref{Eq5: Regret bound for OCO with memory (Unknown case)} over $i$ and $t$, then by combining them with Equation \eqref{Eq3: Regret bound for OCO with memory (Unknown case)} we get
    \begin{equation}\label{Eq6: Regret bound for OCO with memory (Unknown case)}
    \resizebox{\linewidth}{!}
    {$
    \begin{aligned}
        &\sum_{t=T_0+1}^T \sum_i f_{i,t}(\Mb_{j,t-H-1:t}|\widehat{\Ab}_j, \widehat{\Bb}_j, \{\hat{\wb}_j\}) - f_{i,t}(\Mb^*|\widehat{\Ab}_i, \widehat{\Bb}_i, \{\hat{\wb}_i\})\\
        \leq &O\bigg(\eta Tm H^3(G_c D)^2(\kappa+\epsilon)\kappa^3 (W+\zeta\epsilon)\frac{ \big(\frac{\kappa^3(\kappa+\epsilon) (W+\zeta\epsilon)}{\gamma - 2\kappa^3\epsilon} + (W+\zeta\epsilon)\big)\sqrt{m}}{1-\beta}\bigg)\\
        + &O\bigg(mTG_cD \sqrt{2}\left[\kappa\left(\frac{2\kappa^5\epsilon^{\prime}(W+\zeta\epsilon)}{(\gamma - 2\kappa^3\epsilon)^2} + \frac{4\kappa^8H^2(\kappa+\epsilon)\epsilon^{\prime}(W+\zeta\epsilon)}{(1-\gamma+2\kappa^3\epsilon)(\gamma - 2\kappa^3\epsilon)} + \frac{(H+1)2\kappa^5\epsilon^{\prime}(W+\zeta\epsilon)}{(1-\gamma+2\kappa^3\epsilon)(\gamma - 2\kappa^3\epsilon)} + \left(\frac{\kappa^2 + (H)(\kappa+\epsilon)2\kappa^5}{\gamma - 2\kappa^3\epsilon}\right)(2\zeta\epsilon)\right) + \frac{4\kappa^3\zeta\epsilon}{\gamma}\right]\bigg)\\
        + &O\big(\eta T m\sqrt{m}G_cDd^3H\frac{\big(\frac{\kappa^3(\kappa+\epsilon) (W+\zeta\epsilon)}{\gamma - 2\kappa^3\epsilon} + (W+\zeta\epsilon)\big)^2}{1-\beta}\big) + O\bigg(\eta Tm \big(G_c D d\sqrt{d}\sqrt{H} \big(\frac{\kappa^3(\kappa+\epsilon) (W+\zeta\epsilon)}{\gamma - 2\kappa^3\epsilon} + (W+\zeta\epsilon)\big)\big)^2\bigg) + O(\frac{1}{\eta})\\
        + &O\bigg(\frac{1}{\eta}Tm \big(\frac{\eta G_c D d\sqrt{d}\sqrt{H} \big(\frac{\kappa^3(\kappa+\epsilon) (W+\zeta\epsilon)}{\gamma - 2\kappa^3\epsilon} + (W+\zeta\epsilon)\big)\sqrt{m}}{1-\beta}\big)^2\bigg).
    \end{aligned}
    $}
    \end{equation}
    
    \end{proof}

\begin{proof}[Proof of Theorem \ref{T: Regret bound (unknown case)}]
    For the first $T_0$ iterations, the objective is to collect data and estimate the system. Since there is no learning process, the regret of this phase is $O(T_0)$. During the learning phase, based on how the agent's noise is estimate ($\Ab\xb_{i,t} + \Bb\ub_{i,t} + \wb_t = \widehat{\Ab}_i\xb_{i,t} + \widehat{\Bb}_i\ub_{i,t} + \hat{\wb}_{i,t}$), it's equivalent to say agent $i$ is evolving based on the system $(\widehat{\Ab}_i, \widehat{\Bb}_i, \{\hat{\wb}_i\})$. With this fact and supposing $\Kb^*$ is the optimal linear policy in hindsight, we decompose the individual regret of this phase into the following six terms:
    \begin{equation}\label{Eq1: Regret bound (Unknown case)}
    \resizebox{\linewidth}{!}
    {$
    \begin{aligned}
        &c_{i,t}\big(\xb^{\Kb}_{j,t}(\Mb_{j,T_0+1:t-1}|\widehat{\Ab}_j,\widehat{\Bb}_j,\{\hat{\wb}_j\}), \ub^{\Kb}_{j,t}(\Mb_{j,T_0+1:t}|\widehat{\Ab}_j,\widehat{\Bb}_j,\{\hat{\wb}_j\})\big)  -c_{i,t}\big(\xb^{\Kb^*}_t(0|\Ab,\Bb,\{\wb\}),\ub^{\Kb^*}_t(0|\Ab,\Bb,\{\wb\})\big)\\[10pt]
        =&\left[c_{i,t}\big(\xb^{\Kb}_{j,t}(\Mb_{j,T_0+1:t-1}|\widehat{\Ab}_j,\widehat{\Bb}_j,\{\hat{\wb}_j\}), \ub^{\Kb}_{j,t}(\Mb_{j,T_0+1:t}|\widehat{\Ab}_j,\widehat{\Bb}_j,\{\hat{\wb}_j\})\big)\right.\\
        &~~~~~~~~~~~~~~~~~~~~~~~~~~~~~~~~~~~~\left.- c_{i,t}\big(\Tilde{\xb}^{\Kb}_{j,t}(\Mb_{j,T_0+1:t-1}|\widehat{\Ab}_j,\widehat{\Bb}_j,\{\hat{\wb}_j\}), \Tilde{\ub}^{\Kb}_{j,t}(\Mb_{j,T_0+1:t}|\widehat{\Ab}_j,\widehat{\Bb}_j,\{\hat{\wb}_j\})\big)\right]\\[10pt]
        +&\left[c_{i,t}\big(\Tilde{\xb}^{\Kb}_{j,t}(\Mb_{j,T_0+1:t-1}|\widehat{\Ab}_j,\widehat{\Bb}_j,\{\hat{\wb}_j\}), \Tilde{\ub}^{\Kb}_{j,t}(\Mb_{j,T_0+1:t}|\widehat{\Ab}_j,\widehat{\Bb}_j,\{\hat{\wb}_j\})\big)\right.\\
        &~~~~~~~~~~~~~~~~~~~~~~~~~~~~~~~~~~~~\left.- c_{i,t}\big(\yb^{\Kb}_{j,t}(\Mb_{j,t-H-1:t}|\widehat{\Ab}_j,\widehat{\Bb}_j,\{\hat{\wb}_j\}), \vb^{\Kb}_{j,t}(\Mb_{j,t-H-1:t}|\widehat{\Ab}_j,\widehat{\Bb}_j,\{\hat{\wb}_j\})\big)\right]\\[10pt]
        +&\left[c_{i,t}\big(\yb^{\Kb}_{j,t}(\Mb_{j,t-H-1:t}|\widehat{\Ab}_j,\widehat{\Bb}_j,\{\hat{\wb}_j\}), \vb^{\Kb}_{j,t}(\Mb_{j,t-H-1:t}|\widehat{\Ab}_j,\widehat{\Bb}_j,\{\hat{\wb}_j\})\big)\right.\\
        &~~~~~~~~~~~~~~~~~~~~~~~~~~~~~~~~~~~~\left.- c_{i,t}\big(\yb^{\Kb}_{t}(\Mb^*|\widehat{\Ab}_i,\widehat{\Bb}_i,\{\hat{\wb}_i\}), \vb^{\Kb}_{t}(\Mb^*|\widehat{\Ab}_i,\widehat{\Bb}_i,\{\hat{\wb}_i\})\big)\right]\\[10pt]
        +&\left[c_{i,t}\big(\yb^{\Kb}_{t}(\Mb^*|\widehat{\Ab}_i,\widehat{\Bb}_i,\{\hat{\wb}_i\}), \vb^{\Kb}_{t}(\Mb^*|\widehat{\Ab}_i,\widehat{\Bb}_i,\{\hat{\wb}_i\})\big) - c_{i,t}\big(\yb^{\Kb}_{t}(\Mb^*|\Ab,\Bb,\{\wb\}), \vb^{\Kb}_{t}(\Mb^*|\Ab,\Bb,\{\wb\})\big)\right]\\[10pt]
        +&\left[c_{i,t}\big(\yb^{\Kb}_{t}(\Mb^*|\Ab,\Bb,\{\wb\}), \vb^{\Kb}_{t}(\Mb^*|\Ab,\Bb,\{\wb\})\big) - c_{i,t}\big(\xb^{\Kb}_{t}(\Mb^*|\Ab,\Bb,\{\wb\}), \ub^{\Kb}_{t}(\Mb^*|\Ab,\Bb,\{\wb\})\big)\right]\\[10pt]
        +&\left[c_{i,t}\big(\xb^{\Kb}_{t}(\Mb^*|\Ab,\Bb,\{\wb\}), \ub^{\Kb}_{t}(\Mb^*|\Ab,\Bb,\{\wb\})\big) - c_{i,t}\big(\xb^{\Kb^*}_{t}(0|\Ab,\Bb,\{\wb\}), \ub^{\Kb^*}_{t}(0|\Ab,\Bb,\{\wb\})\big)\right],
    \end{aligned}
    $}     
    \end{equation}
    where $\Tilde{\xb}^{\Kb}_{j,t}(\Mb_{j,T_0+1:t-1}|\widehat{\Ab}_j,\widehat{\Bb}_j,\{\hat{\wb}_j\})$ denotes the state assuming $\xb_{j,T_0+1} = 0$ and $\Tilde{\ub}^{\Kb}_{j,t}(\Mb_{j,T_0+1:t}|\widehat{\Ab}_j,\widehat{\Bb}_j,\{\hat{\wb}_j\})$ is the corresponding action. Each term is bounded separately as follows
    
    {\bf Term I:}\\
    Based on Lemma \ref{L: Bounds on states and actions during the exploration phase}, we have $\norm{\xb_{i,T_0+1}}\leq \frac{2\kappa^3 W\sqrt{d_2}}{\gamma}, \forall i$. Then based on Lemma \ref{L: Diff between resulting states based on different initial states}, we get
    \begin{equation}\label{Eq2: Regret bound (Unknown case)}
    \begin{split}
        &\norm{\xb^{\Kb}_{j,t}(\Mb_{j,T_0+1:t-1}|\widehat{\Ab}_j,\widehat{\Bb}_j,\{\hat{\wb}_j\}) - \Tilde{\xb}^{\Kb}_{j,t}(\Mb_{j,T_0+1:t-1}|\widehat{\Ab}_j,\widehat{\Bb}_j,\{\hat{\wb}_j\})}\\
        \leq &(\widehat{\Ab}_j - \widehat{\Bb}_j\Kb)^{t-T_0-1}\frac{2\kappa^3 W\sqrt{d_2}}{\gamma}\leq \kappa^2(1-\gamma+2\kappa^3\epsilon)^{t-T_0-1}\frac{2\kappa^3 W\sqrt{d_2}}{\gamma},
    \end{split}    
    \end{equation}
    and 
    \begin{equation}\label{Eq3: Regret bound (Unknown case)}
    \begin{split}
        &\norm{\ub^{\Kb}_{j,t}(\Mb_{j,T_0+1:t-1}|\widehat{\Ab}_j,\widehat{\Bb}_j,\{\hat{\wb}_j\}) - \Tilde{\ub}^{\Kb}_{j,t}(\Mb_{j,T_0+1:t-1}|\widehat{\Ab}_j,\widehat{\Bb}_j,\{\hat{\wb}_j\})}\\
        \leq &\kappa^3(1-\gamma+2\kappa^3\epsilon)^{t-T_0-1}\frac{2\kappa^3 W\sqrt{d_2}}{\gamma}.
    \end{split}    
    \end{equation}
    Then with Assumption $\ref{A1}$, the function difference is upper bounded as follows
    \begin{equation}\label{Eq4: Regret bound (Unknown case)}
    \resizebox{\linewidth}{!}
    {$
    \begin{aligned}
        &c_{i,t}\big(\xb^{\Kb}_{j,t}(\Mb_{j,T_0+1:t-1}|\widehat{\Ab}_j,\widehat{\Bb}_j,\{\hat{\wb}_j\}), \ub^{\Kb}_{j,t}(\Mb_{j,T_0+1:t}|\widehat{\Ab}_j,\widehat{\Bb}_j,\{\hat{\wb}_j\})\big) - c_{i,t}\big(\xb^{\Kb^*}_t(0|\Ab,\Bb,\{\wb\}),\ub^{\Kb^*}_t(0|\Ab,\Bb,\{\wb\})\big)\\
        \leq &G_c D \sqrt{2}\kappa^3(1-\gamma+2\kappa^3\epsilon)^{t-T_0-1}\frac{2\kappa^3 W\sqrt{d_2}}{\gamma},
    \end{aligned}
    $}    
    \end{equation}
    where $D = \frac{(W+\zeta\epsilon)(\kappa^2 + 2H(\kappa+\epsilon)\kappa^5)}{(\gamma - 2\kappa^3\epsilon)(1-\kappa^2(1-\gamma+2\kappa^3\epsilon)^{(H+1)})} + \frac{2(W+\zeta\epsilon)\kappa^3}{\gamma}$.
    Summing Equation \eqref{Eq4: Regret bound (Unknown case)} over $i$ and $t$, we have
    \begin{equation}\label{Eq5: Regret bound (Unknown case)}
        \sum_{t=T_0+1}^T\sum_{i=1}^m\text{Term I}\leq \frac{2\sqrt{2}G_cD\kappa^6W\sqrt{d_2}}{\gamma(\gamma - 2\kappa^3\epsilon)}.
    \end{equation}
    \\
    {\bf Terms II \& V:}\\
    Based on Lemma \ref{L: Upper bounds of (surrogate)states and (surrogate)actions and the difference}, we have for any $t\geq T_0+1$
    \begin{equation}\label{Eq6: Regret bound (Unknown case)}
    \begin{split}
        &\norm{\Tilde{\xb}^{\Kb}_{j,t}(\Mb_{j,T_0+1:t-1}|\widehat{\Ab}_j,\widehat{\Bb}_j,\{\hat{\wb}_j\}) - \yb^{\Kb}_{j,t}(\Mb_{j,t-H-1:t}|\widehat{\Ab}_j,\widehat{\Bb}_j,\{\hat{\wb}_j\})}\leq \kappa^2(1-\gamma+2\kappa^3\epsilon)^{H+1}D,
    \end{split}
    \end{equation}
    and
    \begin{equation}\label{Eq7: Regret bound (Unknown case)}
    \begin{split}
        &\norm{\Tilde{\ub}^{\Kb}_{j,t}(\Mb_{j,T_0+1:t}|\widehat{\Ab}_j,\widehat{\Bb}_j,\{\hat{\wb}_j\}) - \vb^{\Kb}_{j,t}(\Mb_{j,t-H-1:t}|\widehat{\Ab}_j,\widehat{\Bb}_j,\{\hat{\wb}_j\})}\leq \kappa^3(1-\gamma+2\kappa^3\epsilon)^{H+1}D.
    \end{split}
    \end{equation}
    Then with Assumption $\ref{A1}$, the function difference is upper bounded as follows
    \begin{equation}\label{Eq8: Regret bound (Unknown case)}
    \begin{split}
        &c_{i,t}\big(\Tilde{\xb}^{\Kb}_{j,t}(\Mb_{j,T_0+1:t-1}|\widehat{\Ab}_j,\widehat{\Bb}_j,\{\hat{\wb}_j\}), \Tilde{\ub}^{\Kb}_{j,t}(\Mb_{j,T_0+1:t}|\widehat{\Ab}_j,\widehat{\Bb}_j,\{\hat{\wb}_j\})\big)\\
        - &c_{i,t}\big(\yb^{\Kb}_{j,t}(\Mb_{j,t-H-1:t}|\widehat{\Ab}_j,\widehat{\Bb}_j,\{\hat{\wb}_j\}), \vb^{\Kb}_{j,t}(\Mb_{j,t-H-1:t}|\widehat{\Ab}_j,\widehat{\Bb}_j,\{\hat{\wb}_j\})\big)\\
        \leq &G_cD\sqrt{2} \kappa^3(1-\gamma+2\kappa^3\epsilon)^{H+1}D.
    \end{split}
    \end{equation}
    Summing Equation \eqref{Eq8: Regret bound (Unknown case)} over $i$ and $t$, we have
    \begin{equation}\label{Eq9: Regret bound (Unknown case)}
        \sum_{t=T_0+1}^T\sum_{i=1}^m\text{Term II} \leq mT (1-\gamma+2\kappa^3\epsilon)^{H+1}\sqrt{2}\kappa^3 D^2 G_c.
    \end{equation}
    Based on the similar process, term V is bounded as follows
    \begin{equation}\label{Eq10: Regret bound (Unknown case)}
        \sum_{t=T_0+1}^T\sum_{i=1}^m\text{Term V} \leq mT (1-\gamma)^{H+1}\sqrt{2}\kappa^3 D^2 G_c.
    \end{equation}
    \\
    {\bf Terms III:} By invoking Theorem \ref{T: Regret bound for OCO with memory (Unknown case)}, the upper bound is shown.\\\\
    {\bf Terms IV:}\\
    Based on Lemma \ref{L: Difference between surrogate states (actions) on the real and estiamted systems}, we have for $t\in[T_0+1, T_0+2H+1]$
    \begin{equation}\label{Eq11: Regret bound (Unknown case)}
    \resizebox{\linewidth}{!}
    {
    \begin{minipage}{\linewidth}
    $
    \begin{aligned}
        &c_{i,t}\big(\yb^{\Kb}_{t}(\Mb^*|\widehat{\Ab}_i,\widehat{\Bb}_i,\{\hat{\wb}_i\}), \vb^{\Kb}_{t}(\Mb^*|\widehat{\Ab}_i,\widehat{\Bb}_i,\{\hat{\wb}_i\})\big) - c_{i,t}\big(\yb^{\Kb}_{t}(\Mb^*|\Ab,\Bb,\{\wb\}), \vb^{\Kb}_{t}(\Mb^*|\Ab,\Bb,\{\wb\})\big)\\
    \leq & G_c D \sqrt{\norm{\yb^{\Kb}_{t}(\Mb^*|\widehat{\Ab}_i,\widehat{\Bb}_i,\{\hat{\wb}_i\}) - \yb^{\Kb}_{t}(\Mb^*|\Ab,\Bb,\{\wb\})}^2 + \norm{\vb^{\Kb}_{t}(\Mb^*|\widehat{\Ab}_i,\widehat{\Bb}_i,\{\hat{\wb}_i\}) - \vb^{\Kb}_{t}(\Mb^*|\Ab,\Bb,\{\wb\})}^2}\\
    \leq & G_cD \sqrt{2}\kappa\left(\frac{2\kappa^5\epsilon(W+\zeta\epsilon)}{(\gamma - 2\kappa^3\epsilon)^2} + \frac{4\kappa^8H^2(\kappa+\epsilon)\epsilon(W+\zeta\epsilon)}{(1-\gamma+2\kappa^3\epsilon)(\gamma - 2\kappa^3\epsilon)} + \frac{(H+1)2\kappa^5\epsilon(W+\zeta\epsilon)}{(1-\gamma)\gamma } + \left(\frac{\kappa^2 + (H)2\kappa^6}{\gamma}\right)(2W + \zeta\epsilon)\right)\\
    +&G_cD\sqrt{2}\frac{2\kappa^3(2W+\zeta\epsilon)}{\gamma},
    \end{aligned}
    $    
    \end{minipage}
    }
    \end{equation}
    where the last inequality is due to the fact that $\norm{\hat{\wb}_{i,t} - \wb_t}\leq 2W + \zeta\epsilon$ for $t< T_0+2H+1$ and 
    \begin{equation*}
    \begin{split}
        &\vb^{\Kb}_{t}(\Mb^*|\widehat{\Ab}_i, \widehat{\Bb}_i,\{\hat{\wb}_i\}) - \vb^{\Kb}_{t}(\Mb^*|\Ab, \Bb,\{\wb\})\\
        = &-\Kb\big(\yb^{\Kb}_{t}(\Mb^*|\widehat{\Ab}_i, \widehat{\Bb}_i,\{\hat{\wb}_i\}) - \yb^{\Kb}_{t}(\Mb^*|\Ab, \Bb,\{\wb\})\big) + \sum_{k=1}^H \Mb^{*[k-1]}(\hat{\wb}_{i,t-k} - \wb_{t-k}).
    \end{split}    
    \end{equation*}
    Similiarly, for $t\geq T_0+2H+2$ we have
    \begin{equation}\label{Eq12: Regret bound (Unknown case)}
    \resizebox{\linewidth}{!}
    {
    \begin{minipage}{\linewidth}
    $
    \begin{aligned}
        &c_{i,t}\big(\yb^{\Kb}_{t}(\Mb^*|\widehat{\Ab}_i,\widehat{\Bb}_i,\{\hat{\wb}_i\}), \vb^{\Kb}_{t}(\Mb^*|\widehat{\Ab}_i,\widehat{\Bb}_i,\{\hat{\wb}_i\})\big) - c_{i,t}\big(\yb^{\Kb}_{t}(\Mb^*|\Ab,\Bb,\{\wb\}), \vb^{\Kb}_{t}(\Mb^*|\Ab,\Bb,\{\wb\})\big)\\
    \leq & G_cD \sqrt{2}\kappa\left(\frac{2\kappa^5\epsilon(W+\zeta\epsilon)}{(\gamma - 2\kappa^3\epsilon)^2} + \frac{4\kappa^8H^2(\kappa+\epsilon)\epsilon(W+\zeta\epsilon)}{(1-\gamma+2\kappa^3\epsilon)(\gamma - 2\kappa^3\epsilon)} + \frac{(H+1)2\kappa^5\epsilon(W+\zeta\epsilon)}{(1-\gamma)\gamma } + \left(\frac{\kappa^2 + (H)2\kappa^6}{\gamma}\right) \zeta\epsilon\right)\\
    +&G_cD\sqrt{2}\frac{2\kappa^3(\zeta\epsilon)}{\gamma}.
    \end{aligned}
    $
    \end{minipage}
    }
    \end{equation}
    With Equations \eqref{Eq11: Regret bound (Unknown case)} and \eqref{Eq12: Regret bound (Unknown case)}, we get
    \begin{equation}
    \resizebox{\linewidth}{!}
    {
    \begin{minipage}{\linewidth}
    $
    \begin{aligned}
        &\sum_{t=T_0+1}^T\sum_{i=1}^m\text{Term IV}\\
        \leq&TmG_cD \sqrt{2}\kappa\left(\frac{2\kappa^5\epsilon(W+\zeta\epsilon)}{(\gamma - 2\kappa^3\epsilon)^2} + \frac{4\kappa^8H^2(\kappa+\epsilon)\epsilon(W+\zeta\epsilon)}{(1-\gamma+2\kappa^3\epsilon)(\gamma - 2\kappa^3\epsilon)} + \frac{(H+1)2\kappa^5\epsilon(W+\zeta\epsilon)}{(1-\gamma)\gamma } + \left(\frac{\kappa^2 + (H)2\kappa^6}{\gamma}\right) \zeta\epsilon\right)\\
    +&TmG_cD\sqrt{2}\frac{2\kappa^3(\zeta\epsilon)}{\gamma} + (2H+1)G_cD\sqrt{2}\big(\frac{2W(\kappa^2 + H2\kappa^6)}{\gamma} + \frac{4\kappa^3W}{\gamma}\big).
    \end{aligned}
    $    
    \end{minipage}
    }
    \end{equation}
    {\bf Terms VI:} By invoking Lemma \ref{L: Approximate a linear policy using time-invariant DFC policy}, the upper bound is shown.\\\\
    Combining the separate upper bounds above, we have
    \begin{equation*}
    \resizebox{\linewidth}{!}
    {
    \begin{minipage}{\linewidth}
    \begin{align*}
        &J_T^j(\mathcal{A}) - \min_{\Kb\in\mathcal{K}}J_T(\Kb)\\
        \leq &O(T_0) + O(\frac{2\sqrt{2}G_cD\kappa^6W\sqrt{d_2}}{\gamma(\gamma - 2\kappa^3\epsilon)})  + O(mT (1-\gamma+2\kappa^3\epsilon)^{H+1}\sqrt{2}\kappa^3 D^2 G_c) + O(mT (1-\gamma)^{H+1}\sqrt{2}\kappa^3 D^2 G_c)\\
        + &O\bigg(\eta Tm H^3(G_c D)^2(\kappa+\epsilon)\kappa^3 (W+\zeta\epsilon)\frac{ \big(\frac{\kappa^3(\kappa+\epsilon) (W+\zeta\epsilon)}{\gamma - 2\kappa^3\epsilon} + (W+\zeta\epsilon)\big)\sqrt{m}}{1-\beta}\bigg)\\
        + &O\bigg(mTG_cD \sqrt{2}\left[\kappa\left(\frac{2\kappa^5\epsilon^{\prime}(W+\zeta\epsilon)}{(\gamma - 2\kappa^3\epsilon)^2} + \frac{4\kappa^8H^2(\kappa+\epsilon)\epsilon^{\prime}(W+\zeta\epsilon)}{(1-\gamma+2\kappa^3\epsilon)(\gamma - 2\kappa^3\epsilon)} + \frac{(H+1)2\kappa^5\epsilon^{\prime}(W+\zeta\epsilon)}{(1-\gamma+2\kappa^3\epsilon)(\gamma - 2\kappa^3\epsilon)} + \left(\frac{\kappa^2 + H(\kappa+\epsilon)2\kappa^5}{\gamma - 2\kappa^3\epsilon}\right)(2\zeta\epsilon)\right) + \frac{4\kappa^3\zeta\epsilon}{\gamma}\right]\bigg)\\
        + &O\big(\eta T m\sqrt{m}G_cDd^3H\frac{\big(\frac{\kappa^3(\kappa+\epsilon) (W+\zeta\epsilon)}{\gamma - 2\kappa^3\epsilon} + (W+\zeta\epsilon)\big)^2}{1-\beta}\big) + O\bigg(\eta Tm \big(G_c D d\sqrt{d}\sqrt{H} \big(\frac{\kappa^3(\kappa+\epsilon) (W+\zeta\epsilon)}{\gamma - 2\kappa^3\epsilon} + (W+\zeta\epsilon)\big)\big)^2\bigg) + O(\frac{1}{\eta})\\
        + &O\bigg(\frac{1}{\eta}Tm \big(\frac{\eta G_c D d\sqrt{d}\sqrt{H} \big(\frac{\kappa^3(\kappa+\epsilon) (W+\zeta\epsilon)}{\gamma - 2\kappa^3\epsilon} + (W+\zeta\epsilon)\big)\sqrt{m}}{1-\beta}\big)^2\bigg)\\
        +&O\left(TmG_cD \sqrt{2}\kappa\left(\frac{2\kappa^5\epsilon(W+\zeta\epsilon)}{(\gamma - 2\kappa^3\epsilon)^2} + \frac{4\kappa^8H^2(\kappa+\epsilon)\epsilon(W+\zeta\epsilon)}{(1-\gamma+2\kappa^3\epsilon)(\gamma - 2\kappa^3\epsilon)} + \frac{(H+1)2\kappa^5\epsilon(W+\zeta\epsilon)}{(1-\gamma)\gamma } + \left(\frac{\kappa^2 + (H)2\kappa^6}{\gamma}\right) \zeta\epsilon\right)\right)\\
        +&O\left(TmG_cD\sqrt{2}\frac{2\kappa^3(\zeta\epsilon)}{\gamma} + (2H+1)G_cD\sqrt{2}\big(\frac{2W(\kappa^2 + H2\kappa^6)}{\gamma} + \frac{4\kappa^3W}{\gamma}\big)\right)\\
        +&O(TmG_c D\frac{6WH \kappa_B^2\kappa^6(1-\gamma)^{H-1}}{\gamma}),
    \end{align*}   
    \end{minipage}
    }
    \end{equation*}
    where by choosing $\eta = $, $T_{collect} = $ and $T_{exchange} = O(\log T^{\rho})$ with $\rho$ large enough, and knowing that $\epsilon' = O(\epsilon)$ and $\epsilon = O\big(\frac{d_2W\sqrt{d_1 q}\kappa^{11/2}}{\gamma}\sqrt{\frac{\log\big(d_1 d_2 (q+1)m\delta^{-1}\big)}{m\big(T_{collect}-(q+1)\big)}}\big)$ from Theorem \ref{T: System identification}, the result is proved.
\end{proof}


\end{document}